\documentclass[11pt]{amsart}
\title[Morgan's mixed Hodge structures and nonabelian Hodge structures]
{Morgan's mixed Hodge structures and nonabelian Hodge structures}

\author{Hisashi Kasuya}

\usepackage{fullpage}
\usepackage{amssymb}
\usepackage{amsmath}
\usepackage{amscd}
\usepackage{amstext}
\usepackage{amsfonts}
\usepackage[all]{xy}

\theoremstyle{plain}
\newtheorem{construction}{Construction}[subsection]
\theoremstyle{plain}

\theoremstyle{plain}

\theoremstyle{plain}

\theoremstyle{plain}
\newtheorem{theorem}{Theorem}[subsection] 
\theoremstyle{remark}
\newtheorem{remark}[theorem]{Remark}
\theoremstyle{remark}
\newtheorem{Important note}[theorem]{Important note}
\theoremstyle{Main result}
\newtheorem{main result}{Main result}
\theoremstyle{lemma}
\newtheorem{lemma}[theorem]{Lemma}
\theoremstyle{definition}
\newtheorem{definition}[theorem]{Definition}
\theoremstyle{proposition}
\newtheorem{proposition}[theorem]{Proposition}
\theoremstyle{corollary}

\theoremstyle{remark}
\newtheorem{example}[theorem]{Example}
\theoremstyle{plain}

\theoremstyle{problem}

\theoremstyle{conclusion}

\newtheorem*{cond(V)}{Condition (V)}

\address[Hisashi Kasuya]{
Department of Mathematics, Graduate School of Science, Osaka University, Osaka,  Japan. }
\email{kasuya@math.sci.osaka-u.ac.jp}


\newcommand{\C}{\mathbb{C}}
\newcommand{\R}{\mathbb{R}}

\newcommand{\g}{\frak{g}}

\newcommand{\V}{\mathcal V}
\newcommand{\W}{\mathcal W}

\begin{document} 

\maketitle
\begin{abstract}
We refine the Morgan's work on  mixed Hodge structures on Sullivan's  $1$-minimal models by using non-abelian Hodge theory.
As an application, we give explicit representatives   of 
real unipotent variations of mixed Hodge structures over compact K\"ahler manifolds.

\end{abstract}

\section{Introduction}

\subsection{Geometric motivation:Non-abelian Hodge structures on compact K\"ahler manifolds}
Let $M$ be a compact K\"ahler manifold.
It is well known that the cohomology of $M$ admits a functorial Hodge structure.
We can explicitly  see the real  Hodge structure on the first cohomology $H^{1}(M,\R)$.
On the de Rham complex $A^{\ast}(M)$ with the bigrading $A^{n}(M)\otimes \C=\bigoplus_{p+q=n} A^{p,q}(M)$, we 
consider the differential operators $d$, $d^{c}$, $\partial$ and $\bar{\partial}$.
On each $A^{n}(M)$, we have  the $dd^{c}$-Lemma:
\[{\rm im}d\cap {\rm ker}d^{c}={\rm ker}d\cap {\rm im}d^{c}={\rm im}dd^{c}
\]
 and  on each $A^{p,q}(M)$,
we have the $\partial\bar\partial$-Lemma:
\[{\rm im}\partial\cap {\rm ker}\bar\partial={\rm ker}\partial\cap {\rm im}\bar\partial={\rm im}\partial\bar\partial
\]
(see \cite{DGMS}).
Applying these relations on $1$-forms, we have 
\[H^{1}(M,\R)\cong {\rm ker}d\cap {\rm ker}d^{c}\qquad {\rm and} \qquad
H^{1}(M,\C)\cong ({\rm ker}d\cap {\rm ker}d^{c})_{\C}={\rm ker}\partial\cap{\rm ker}\bar\partial.
\]
Thus, we can see the  real  Hodge structure on  $H^{1}(M,\R)$ at the level of explicit differential forms.
The purpose of this paper is to extend this observation to non-abelian Hodge theory in terms of \cite{Sim}, \cite{Ara}.

In \cite{Ara}, Arapura formulates non-abelian mixed Hodge structures
 as a refinement of Tannakian considerations of  non-abelian Hodge theory as in \cite{Sim}.
Non-abelian mixed Hodge structures appear as invariants of "pointed" varieties $(M,x)$.
We notice that unlike usual Hodge theory  the dependence on the base point is very important for non-abelian Hodge theory (see \cite[Section 5]{Sim}).
Arapura suggest that variations of mixed Hodge structures  with fiber functors associated with points should be considered as non-abelian mixed Hodge structures.
Our  main geometric progress  of this paper is stated as follows.
"Morgan's mixed Hodge structures associated with pointed compact K\"ahler manifolds $(M,x)$ provide explicit representatives of unipotent variations of mixed Hodge structures over $M$ with fiber functors associated with points $x$."

\subsection{Algebraic theory: Morgan's mixed Hodge structures and nonabelian Hodge structures}
For the above geometric motivation, we study the pure algebraic problem on the  invariance of Morgan's mixed Hodge structures.
In \cite{Del}, Deligne establishes the theory of mixed Hodge complexes.
They are homological algebraic  objects giving functorial  canonical mixed Hodge structures on the cohomology.
In \cite{Mor}, Morgan consideres a multiplicative version of mixed Hodge complexes called mixed Hodge diagrams.
These are algebraic models of the (logarithmic) de Rham complexes of (quasi)-K\"ahler manifolds.
Morgan constructs various non-canonical mixed Hodge structures for   any mixed Hodge diagram so-called Morgan mixed Hodge structures on Sullivan's $1$-minimal model.
The existence of Morgan mixed Hodge structures is  very important progress.
But,  by the  ambiguity of choosing,  
it may be  difficult to consider such mixed Hodge structures themselves as functorial invariants of  mixed Hodge diagrams.
The purpose of this paper is to prove that Morgan mixed Hodge structures of "augmented" mixed Hodge diagrams represent functorial invariants of augmented mixed Hodge diagrams.

Our main invariants  are  non-abelian mixed Hodge structures associated with augmented mixed Hodge diagrams
which are analogues of  corresponding  pointed varieties $(M,x)$ to  unipotent variations of mixed Hodge structures over $M$ with fiber functors associated with points $x$.
Our main progress is to prove that "Morgan mixed Hodge structures   of  augmented mixed Hodge diagrams represent such  non-abelian mixed Hodge structures".

\begin{remark}
In \cite{Mor}, Morgan also constructs mixed Hodge structures on (full) minimal models of "simply connected" mixed Hodge diagrams.
In \cite{Cir} (also \cite{CG}), refining Morgan's construction, it is shown that  mixed Hodge structures on  minimal models can be seen as functorial invariants of  mixed Hodge diagrams without augmentations (see \cite[Theorem 5.14.]{Cir}).
Mixed Hodge structures on (full) minimal models of simply connected mixed Hodge diagrams are very different from 
 mixed Hodge structures on  $1$-minimal model of non-simply connected mixed Hodge diagrams (see \cite{Morco}).
 This is  related to the importance of the dependence of the base point for non-simply connected varieties.
This matter  appears  on the differential graded Lie algebra version of mixed Hodge diagrams (see \cite{Le} and \cite{Le2}).
\end{remark}

 \noindent{\sl Acknowledgments.}
The author would like to  thank  the  anonymous reviewers for their careful reading  and   many valuable comments and suggestions.

\section{Morgan mixed Hodge structures of augmented  $\mathbb K$-mixed-Hodge diagram}
\subsection{A  review of  Morgan mixed Hodge structures}
Let  $\mathbb K$ be a sub-field of $\C$.

An {\em$\mathbb K$-mixed  Hodge structure} on a finite-dimensional $\mathbb K$-vector space $V$ is a pair $(W_{\ast},F^{\ast})$
such that:
\begin{enumerate}
\item $W_{\ast}$ is an increasing filtration on $V$.
\item $F^{\ast}$ is a decreasing filtration  on $V_{\C}$  such that
the filtration on $Gr_{n}^{W} V_{\C}$ induced by $F^{\ast}$ is an $\mathbb K$-Hodge structure of weight $n$.
\end{enumerate}
We call $W_{\ast}$ the weight filtration and $F^{\ast}$ the Hodge filtration.
A morphism of $\mathbb K$-mixed Hodge structures is a $\mathbb K$-linear map which is compatible with the filtrations $W_{\ast}$ and $F^{\ast}$.
This morphism preserves the canonical bigradings and in fact we can say that this morphism is strictly compatible with the filtrations $W_{\ast}$ and $F^{\ast}$.
Consider the category ${\mathcal MHS}_{\mathbb K}$ of  $\mathbb K$-mixed Hodge structures i.e.  objects are finite-dimensional $\mathbb K$-vector spaces with  $\mathbb K$-mixed Hodge structures and morphisms are morphisms of $\mathbb K$-mixed Hodge structures.
It is known that 
${\mathcal MHS}_{\mathbb K}$ is an abelian category (see \cite{Del}).

We review homotopy theory of (graded-commutative) $\mathbb K$-{\em differential graded algebra} (shortly  $\mathbb K$-{\em DGA}).
A $\mathbb K$-DGA $A^{\ast}$ is {\em cohomologically connected } if   $A^{\ast}$ is unital and the unit map ${\mathbb K}\to A^{\ast}$ induces an isomorphism ${\mathbb K}\cong H^{0}(A^{\ast})$.
In this paper, we always assume every $\mathbb K$-DGA is  cohomologically connected.

Let $A^{\ast}$ and $B^{\ast}$ be $\mathbb K$-DGAs.
If  a morphism of graded algebra $\varphi : A^{\ast} \rightarrow B^{\ast}$ satisfies $d\circ \varphi =\varphi \circ d$, we call $\varphi $ a morphism of $\mathbb K$-DGAs.
If a morphism of $\mathbb K$-DGAs induces a cohomology isomorphism, we call it a {\em quasi-isomorphism}.
If a morphism of $\mathbb K$-DGAs induces
 isomorphisms on $0$-th and first cohomology and an injection on the second cohomology, we call it a {\em $1$-quasi-isomorphism}.
Define the $\mathbb K$-DGA $(t,dt)$ as the tensor product of the ring of $\mathbb K$-polynomials on $t$ with the exterior algebra of $\langle dt\rangle$ so that $t$ is of degree $0$, $d(t)=dt$ and $d(dt)=0$.

Let $A^{\ast}$ be a $\mathbb K$-DGA.
We consider the $\mathbb K$-DGA $A^{\ast}\otimes (t,dt)$.
Each element of $A^{n}\otimes (t,dt)$ can be written as
$\sum (a_{i}t^{i}+b_{i}t^{i}dt )$
with $a_{i}\in A^{n}$ and $b_{i}\in A^{n-1}$.
We can define the  "integrations" $\int^{1}_{0}:A^{\ast}\otimes (t,dt)\to A^{\ast}$ and $\int^{t}_{0}:A^{\ast}\otimes (t,dt)\to A^{\ast}\otimes  (t,dt)$ so that
\[\int_{0}^{1}\sum (a_{i}t^{i}+b_{i}t^{i}dt )=(-1)^{n-1}\sum \frac{b_{i}}{i+1}\qquad
{\rm and}\qquad
\int_{0}^{t}\sum (a_{i}t^{i}+b_{i}t^{i}dt )=(-1)^{n-1}\sum \frac{b_{i} t^{i+1}}{i+1}.
\]

For  two morphisms  $\phi_{0}$ and $\phi_{1}$  from  $A^{\ast}$ to a $\mathbb K$-DGA $B^{\ast}$,
 a homotopy from $\phi_{0}$ to  $\phi_{1}$ is a morphism $H:A^{\ast}\to B^{\ast}\otimes  (t,dt)$ so that 
 for any $x\in A^{\ast}$ we have  
\[H(x)\vert_{t=0}= \phi_{0}(x)\,\,\,\,\,\,
{\rm and}\,\,\,\,\,\,  H(x)\vert_{t=1}= \phi_{1}(x).\]
See \cite[Chapter 11]{GMo} for basic  properties of homotopies and integrations.

A $\mathbb K$-DGA $\mathcal M^{\ast}$  is {\em $1$-minimal} if $\mathcal M^{\ast}=\bigcup\mathcal M^{\ast}_{i}$ for a sequence of sub-DGAs
\[{\mathbb K}=\mathcal M^{\ast}_{0}\subset \mathcal M^{\ast}_{1}\subset \mathcal M^{\ast}_{2}\subset\dots
\]
such that  $\mathcal M_{i+1}^{\ast}=\mathcal M_{i}^{\ast}\otimes \bigwedge \V_{i+1}$,
$\V_{i+1}$ is a graded vector space of degree $1$ and $d\V_{i+1}\subset   \mathcal M_{i}^{2}$.
For each $i$, we have 
\[\mathcal M_{i}^{\ast}=\bigwedge ( \V_{1}\oplus \dots \oplus \V_{i}). 
\]
The dual space  ${\frak n}_{i}=\V^{\ast}_{1}\oplus \dots \oplus \V^{\ast}_{i}$ is a Lie algebra.
We can easily check that this Lie algebra is nilpotent.
We consider the pronilpotent Lie algebra $\frak n= \varprojlim{\frak n}_{i}$.
We call this the {\em dual Lie algebra} of a $1$-minimal $\mathbb K$-DGA $\mathcal M^{\ast}$.

\begin{definition}
For a $\mathbb K$-DGA $A^{\ast}$, a  {\em $1$-minimal model} of $A^{\ast}$ is a 
 $1$-minimal DGA $\mathcal M^{\ast}$ admitting a $1$-quasi-isomorphism $\phi: \mathcal M^{\ast}\to A^{\ast}$.
\end{definition}

\begin{theorem}[\cite{DGMS,GMo}]\label{exunis}
For any (cohomologically connected) $\mathbb K$-DGA  $A^{\ast}$, a $1$-minimal model of $A^{\ast}$ exists and it is unique up to isomorphism of $\mathbb K$-DGA.
More precisely, for any two $1$-minimal models $\mathcal M^{\ast}$ and $\mathcal N^{\ast}$ with $1$-quasi-isomorphisms $\phi:{\mathcal M}^{\ast}\to A^{\ast}$ and $\psi:{\mathcal N}^{\ast}\to A^{\ast}$,
there exists an isomorphism ${\mathcal I}:{\mathcal M}^{\ast}\to {\mathcal N}^{\ast}$ and a homotopy ${\mathcal H}:{\mathcal M}^{\ast}\to A^{\ast}\otimes (t, dt)$ from $\psi\circ {\mathcal I}$ to $\phi$.
\end{theorem}
By the uniqueness, we can say that if there exists a $1$-quasi-isomorphism between two   $\mathbb K$-DGAs  $A^{\ast}$ and $B^{\ast}$, then they have the same $1$-minimal model.
Without the homotopy argument, a map $\phi: \mathcal M^{\ast}\to A^{\ast}$ is not unique.

Let $A^{\ast}$ be a $\mathbb K$-DGA
 and $\mathcal M^{\ast}$  the $1$-minimal model of $A^{\ast}$  with a $1$-quasi-isomorphism $\phi:\mathcal M^{\ast}\to A^{\ast}$.
Then there is  a {\em canonical sequence} for $\phi:\mathcal M^{\ast}\to A^{\ast}$:
\[{\mathcal M}^{\ast}_{1}\subset {\mathcal M}^{\ast}_{2}\subset\dots 
\]
such that:
\begin{enumerate}
\item $\mathcal M^{\ast}=\bigcup_{i=1}^{\infty} \mathcal M^{\ast}_{i}$.
\item ${\mathcal M}^{\ast}_{1}=\bigwedge {\mathcal V}_{1}$ with the trivial differential such that $\phi$ induces an isomorphism 
${\mathcal V}_{1}\to H^{1}(A^{\ast})$.
\item Consider the map $\phi_{1}:   {\mathcal V}_{1}\wedge {\mathcal V}_{1}\to H^{2}(A^{\ast})$ induced by $\phi$.
We have
${\mathcal M}^{\ast}_{2}={\mathcal M}^{\ast}_{1}\otimes \bigwedge {\mathcal V}_{2}$ such that the differential $d$ induces an isomorphism ${\mathcal V}_{2}\to {\rm Ker}\, \phi_{1}$.
\item Consider the map $\phi_{n}:   H^{2}({\mathcal M}^{\ast}_{n})\to H^{2}(A^{\ast})$ induced by $\phi$ for each integer $n$.
We have
${\mathcal M}^{\ast}_{n+1}={\mathcal M}^{\ast}_{n}\otimes \bigwedge {\mathcal V}_{n+1}$ such that the differential $d:{\mathcal V}_{n+1}\to {\mathcal M}^{2}_{n}$ induces an isomorphism ${\mathcal V}_{n+1}\to {\rm Ker}\, \phi_{n}$.

\end{enumerate}

A {\em filtered $\mathbb K$-DGA}  is a DGA $A^{*}$ with a increasing filtration $W_{*}$ so that the differential and the multiplication are compatible with $W_{*}$.
A {\em bifiltered $\mathbb K$-DGA}  is a DGA $A^{*}$ with a increasing filtration $W_{*}$ and a decreasing filtration $F^{\ast}$ so that the differential and the multiplication are compatible with $W_{*}$ and $F^{\ast}$.
For a filtered $\mathbb K$-DGA $(A^{\ast},W_{\ast})$, we define the filtration $W_{*}$ on $A^{\ast}\otimes (t,dt)$ by
$W_{i}(A^{\ast}\otimes (t,dt))=W_{i}(A^{\ast})\otimes (t,dt)$.

A {\em filtered  minimal $\mathbb K$-DGA} is a minimal $\mathbb K$-DGA ${\mathcal M}^{\ast}$ with a increasing filtration $W_{*}$ so that   the differential and the multiplication are strictly  compatible with $W_{*}$.
A {\em filtered  minimal model} of filtered $\mathbb K$-DGA $(A^{\ast},W_{\ast})$ is a filtered  minimal $\mathbb K$-DGA 
$({\mathcal M}^{\ast},W_{\ast})$ admitting a $1$-quasi-isomorphism $\phi: {\mathcal M}^{\ast}\to A^{\ast}$ which is compatible with the filtrations $W_{\ast}$.

\begin{definition}[{\cite[Definition 3.5]{Mor}}]\label{MOMHD}
An $\mathbb K$-{\em mixed-Hodge diagram} ${\mathcal D}=\{(A^{*}, W_{*})\to^{\iota} (B^{*}, W_{*},F^{*})\}$ is a pair of filtered $\mathbb K$-DGA $(A^{*}, W_{*})$ and bifiltered $\C$-DGA $(B^{*}, W_{*},F^{*})$  and filtered DGA map $\iota:(A^{*}_{\C},W_{*})\to (B^{*},W_{*})$  such that:
\begin{enumerate}
\item $\iota$  induces an isomorphism  $\iota^{*}:\,_{W}E^{*,*}_{1}(A^{*}_{\C})\to \,_{W}E^{*,*}_{1}(B^{*})$ where $ \,_{W}E_{*}^{*,*}(\cdot)$ is the spectral sequence for the decreasing filtration $W^{*}=W_{-*}$.
\item The differential $d_{0}$ on $\,_{W}E^{*,*}_{0}(B^{*})$ is strictly compatible with the filtration induced by $F$.
\item The filtration on $\,_{W}E_{1}^{p,q}(B^{*})$ induced by $F$ is an $\mathbb K$-Hodge structure of weight $q$ on $\iota^{*}(\,_{W}E^{*,*}_{1}(A^{*}))$.

\item $W_{-1}(A^{\ast})=0$ and $W_{-1}(B^{\ast})=0$.
\end{enumerate}
\end{definition}
\begin{remark}
This definition is more strict than the definition in \cite{Cir}.
In \cite{Cir}, a mixed-Hodge diagram is defined by a string of quasi-isomorphisms and the non-negativity of weight filtrations  (the fourth condition)  is not assumed.
The main purpose  of this paper is to study real  non-abelian mixed Hodge structures of smooth varieties.
Our definition is suitable for real mixed Hodge theory on smooth varieties (see Section \ref{EXMHD} and Section \ref{cpkah}).
But, it may not be  so for  singular  varieties.

\end{remark}

For a $\mathbb K$-mixed-Hodge diagram ${\mathcal D}=\{(A^{*}, W_{*})\to^{\iota} (B^{*}, W_{*},F^{*})\}$, it is known that the spectral sequence $\,_{F}E^{*,*}_{r}(B^{*})$ for the filtration $F^{\ast}$ degenerates at $E_{1}$-term and it implies that 
the differential $d$ on $B^{*}$ is strictly compatible with the filtration $F^{\ast}$ (see \cite{Del, Mor}).
It is also known that the spectral sequence $\,_{W}E^{*,*}_{r}(A^{*})$ for the filtration $W^{\ast}$ degenerates at $E_{2}$-term.
Define the filtration $W^{\prime}_{*}$ on $A^{*}$ (resp. $B^{*}$)
as $W^{\prime}_{i}(A^{r})=\{x\in W_{i-r}(A^{r})\vert dx\in  W_{i-r-1}(A^{r+1})\}$ (resp.\ $W^{\prime}_{i}(B^{r})=\{x\in W_{i-r}(B^{r})\vert dx\in  W_{i-r-1}(B^{r+1})\}$).
The $E_{2}$-degeneration of $\,_{W}E^{*,*}_{r}(A^{*})$ (resp. $\,_{W}E^{*,*}_{r}(B^{*})$) implies the $E_{1}$-degeneration of the spectral sequence associated with the filtration  $W^{\prime}$.
Hence, the differential $d$ on $A^{*}$ (resp. $B^{*}$) is strictly compatible with the filtration $W^{\prime}$.

\begin{theorem}[{\cite[Theorem 4.3]{Mor}}]\label{midimi}
The filtrations $W^{\prime}_{*}$ and $F^{*}$ induce an $\mathbb K$-mixed Hodge structure  on 
$H^{r}(A^{*})$ via the isomorphism $\iota^{\ast}:H^{r}(A^{*}_{\C})\to H^{r}(B^{*})$.
\end{theorem}

Now we explain Morgan's constructions of mixed Hodge structures on $1$-minimal models.
A {\em bigraded $1$-minimal $\mathbb K$-DGA} is a $1$-minimal $\mathbb K$-DGA $\mathcal N$ with a bigrading 
\[\mathcal N=\bigoplus_{0\le P,Q}({\mathcal N}^{\ast})^{P,Q} 
\]
such that  $({\mathcal N}^{\ast})^{0,0}= {\mathcal N}^{0}=\mathbb K$ and the differential and the multiplication are of type $(0,0)$.
For a bigraded minimal $\mathbb K$-DGA ${\mathcal N}^{\ast}=\bigoplus_{0\le P,Q}({\mathcal N}^{\ast})^{P,Q} $, 
we define the filtrations $W_{\ast}$ and $F^{\ast}$ by
\[W_{i}({\mathcal N}^{\ast})=\bigoplus_{P+Q\le i}({\mathcal N}^{\ast})^{P,Q}, \qquad {\rm and}\qquad F^{r}({\mathcal N}^{\ast})=\bigoplus_{P\ge r}({\mathcal N}^{\ast})^{P,Q}.
\]

\begin{definition}\label{bigmi}
Let ${\mathcal D}=\{(A^{*}, W_{*})\to^{\iota}(B^{*}, W_{*},F^{*})\}$ be an $\mathbb K$-mixed-Hodge diagram.
A bigraded minimal model of ${\mathcal D}$ is a bigraded $1$-minimal $\C$-DGA ${\mathcal N}^{\ast}=\bigoplus_{0\le P,Q}({\mathcal N}^{\ast})^{P,Q} $ which admits a (not necessarily  commutative) diagram:
\[\xymatrix{
B^{\ast}&\ar[l]^{\iota}A^{\ast}_{\C}\ar[r]^{\bar\iota}&\overline{B^{\ast}}\\
&{\mathcal N}^{\ast}\ar[lu]^{\psi} \ar[u]^{\psi^{\prime\prime}}\ar[ru]_{\psi^{\prime}}
}
\]
where $\overline{B^{\ast}}$ is the $\C$-DGA $B^{\ast}$ with the opposite complex structure
and homotopies $ H: {\mathcal N}^{\ast}\to B^{\ast}\otimes (t,dt)$ and  $ H^{\prime}: {\mathcal N}^{\ast}\to \overline{B^{\ast}}\otimes (t,dt)$ from $\iota\circ \psi^{\prime\prime}$ to $\psi$ and $\bar\iota\circ \psi^{\prime\prime}$ to $\psi^{\prime}$
such that for any $(P,Q)$:
\begin{enumerate}
\item  $\psi^{\prime\prime}(({\mathcal N}^{\ast})^{P,Q})\subset W^{\prime}_{P+Q}(A^{\ast}_{\C}),\qquad\qquad
\psi(({\mathcal N}^{\ast})^{P,Q})\subset W_{P+Q}^{\prime}(B^{\ast})\cap F^{P}(B^{\ast}) $
 and  \[
\psi^{\prime}(({\mathcal N}^{\ast})^{P,Q})\subset W^{\prime}_{P+Q}(\overline{B^{\ast}})\cap \overline{F^{Q}(\overline{B^{\ast}})}+\sum_{i\ge 2} W^{\prime}_{P+Q-i}(\overline{B^{\ast}})\cap \overline{F^{Q-i+1}(\overline{B^{\ast}})}\]
\item ${\mathcal H}(({\mathcal N}^{\ast})^{P,Q})\subset W^{\prime}_{P+Q}(B^{\ast}\otimes (t,dt))\qquad {\rm and}  \qquad{\mathcal H}^{\prime}(({\mathcal N}^{\ast})^{P,Q})\subset W^{\prime}_{P+Q}(B^{\ast}\otimes (t,dt))$.
\end{enumerate}

\end{definition}

We notice that $({\mathcal N}^{\ast},W_{\ast})$ is a filtered $1$-minimal model of $(B^{\ast},W^{\prime}_{\ast})$.

\begin{theorem}[\cite{Mor}]\label{Morgan}
Let ${\mathcal D}=\{(A^{*}, W_{*})\to^{\iota}(B^{*}, W_{*},F^{*})\}$ be an $\mathbb K$-mixed-Hodge diagram.
Then:
\begin{enumerate}
\item There exists a filtered $1$-minimal model $({\mathcal M}^{\ast},W_{\ast})$ of the filtered $\mathbb K$-DGA $(A^{\ast},W^{\prime}_{\ast})$ and it is unique up to isomorphism of filtered $\mathbb K$-DGAs
\item There exists a bigraded $1$-minimal model ${\mathcal N}^{\ast}=\bigoplus_{0\le P,Q}({\mathcal N}^{\ast})^{P,Q} $ of ${\mathcal D}$ and it is unique up to isomorphism of bigraded $\C$-DGA.

\item For any $1$-quasi-isomorphism $\phi:{\mathcal M}^{\ast}\to A^{\ast}$  which is compatible with the filtrations $W_{\ast}$ and $W^{\prime}_{\ast}$  and any diagram
 \[\xymatrix{
B^{\ast}&\ar[l]_{\iota}A^{\ast}_{\C}\ar[r]^{\bar\iota}&\overline{B^{\ast}}\\
&{\mathcal N}^{\ast}\ar[lu]^{\psi} \ar[u]^{\psi^{\prime\prime}}\ar[ru]_{\psi^{\prime}}
}
\]
as in Definition \ref{bigmi}, there exists a isomorphism ${\mathcal I}: ({\mathcal N}^{\ast}, W_{\ast})\to ({\mathcal M}^{\ast}, W_{\ast})$ of filtered DGAs and a homotopy ${\mathcal H}:{\mathcal N}^{\ast}\to B^{\ast}\otimes (t,dt)$ which is compatible with the filtrations $W_{\ast}$ and $W^{\prime}_{\ast}$ from $\psi$ to $\iota\circ \phi\circ \mathcal I$.

\item For any isomorphism ${\mathcal I}: ({\mathcal N}^{\ast}, W_{\ast})\to ({\mathcal M}^{\ast},W_{\ast})$ as in the above sentence, 
defining the filtration $F^{\ast}$ on ${\mathcal  M}^{\ast}_{\C}$ by $F^{r}({\mathcal M}^{\ast}_{\C})={\mathcal I}(F^{r}({\mathcal N}^{\ast}))$, $(W_{\ast}, F^{\ast})$ is a $\mathbb K$-mixed-Hodge structure on ${\mathcal  M}^{\ast}$.

\end{enumerate}

\end{theorem}

\begin{remark}
Morgan also show that a $1$-quasi-isomorphism $\phi$ and a  diagram $\xymatrix{
B^{\ast}&\ar[l]_{\iota}A^{\ast}_{\C}\ar[r]^{\bar\iota}&\overline{B^{\ast}}\\
&{\mathcal N}^{\ast}\ar[lu]^{\psi} \ar[u]^{\psi^{\prime\prime}}\ar[ru]_{\psi^{\prime}}
}$ are defined up to homotopy compatible with structures.
%But, in general, there are not canonical choices.
%In the most important case, we can  canonically obtain  them.
%See Section.
\end{remark}

\begin{remark}
An isomorphism ${\mathcal I}:({\mathcal N}^{\ast}, W_{\ast})\to ({\mathcal M}^{\ast},W_{\ast})$   is not unique.
This is very important since the $\mathbb K$-mixed Hodge structure on ${\mathcal  M}^{\ast}$ varies depending on the choice of $\mathcal I$.

\end{remark}

We will argue precise  constructions of an  isomorphism ${\mathcal I}:({\mathcal N}^{\ast}, W_{\ast})\to ({\mathcal M}^{\ast},W_{\ast})$ and a homotopy ${\mathcal H}:({\mathcal N}^{\ast},W_{\ast})\to (B^{\ast}\otimes (t,dt), W^{\prime}_{\ast})$ from $\psi$ to $\iota\circ \phi\circ \mathcal I$ in the next subsection.

\subsection{Morgan mixed Hodge structures of augmented  $\mathbb K$-mixed-Hodge diagram} \label{IHH}
Let ${\mathcal D}=\{(A^{*}, W_{*})\to^{\iota}(B^{*}, W_{*},F^{*})\}$ be an $\mathbb K$-mixed-Hodge diagram.
Additionally we assume that the restriction $d_{\vert B^{0}}:B^{0}\to B^{1}$ of the differential on $B^{0}$ is strictly compatible with the filtration $W_{\ast}$.
We fix a $1$-quasi-isomorphism $\phi:{\mathcal M}^{\ast}\to A^{\ast}$  which is compatible with the filtrations and a diagram 
$\xymatrix{
B^{\ast}&\ar[l]_{\iota}A^{\ast}_{\C}\ar[r]^{\bar\iota}&\overline{B^{\ast}}\\
&{\mathcal N}^{\ast}\ar[lu]^{\psi} \ar[u]^{\psi^{\prime\prime}}\ar[ru]_{\psi^{\prime}}
}$ 
as in Theorem \ref{Morgan}.  
It is known that there is a canonical sequence for $\psi: {\mathcal N}^{\ast}\to B^{\ast}$:
\[{\mathbb C}=\mathcal N^{\ast}_{0}\subset \mathcal N^{\ast}_{1}\subset \mathcal N^{\ast}_{2}\subset\dots,  \,\,\, \, \quad{\mathcal N}^{\ast}_{i+1}={\mathcal N}^{\ast}_{i}\otimes \bigwedge \V_{i+1}
\]
such that  the bigrading ${\mathcal N}^{\ast}=\bigoplus_{0\le P,Q}({\mathcal N}^{\ast})^{P,Q} $ can be restricted to a bigrading on $\V_{i}$ (see \cite[Section 6]{Mor}).
We also fix this sequence.

We take a subspace $C\subset B^{0}$ so that the restriction $d_{\vert C}: C\to d(B^{0})$ of the differential $d$ is bijective.
We will construct canonical ${\mathcal I}$ and ${\mathcal H}$ as in Theorem \ref{Morgan} depending on $C$.
Define $\delta_{C}:d(A^{0})\to C$ as the inverse of $d_{\vert C}: C\to d(B^{0})$.
Then  $\delta_{C}$ is strictly  compatible with the filtrations $W^{\prime}_{\ast}$ and $F^{\ast}$.
For this construction, we use the arguments  in the proof of \cite[Proposition 7.5]{Mor}.

For the fixed canonical sequence as above, 
we will construct ${\mathcal I}$ and ${\mathcal H}$ inductively.
We assume that we  have already defined ${\mathcal I}$ and ${\mathcal H}$  on $\mathcal N_{i}^{\ast}$.
Let $v\in W_{l}(\V_{i+1})$. We consider ${\mathcal I}(dv)\in W_{l}({\mathcal M}^{2})$.
Then we have $a\in W_{l}({\mathcal M}^{1})$ so that $da_{v}={\mathcal I}(dv)$.
Then, $\psi(v)-\iota\circ \phi(a_{v})+\int_{0}^{1}{\mathcal H}(dv)\in W^{\prime}_{l}(B^{1})$ is closed.
Since $\iota\circ \phi: ({\mathcal M}^{\ast},W_{\ast})\to (B^{\ast},W^{\prime}_{\ast})$ is a $1$-quasi-isomorphism,
we can take $a_{v}\in W_{l}({\mathcal M}^{1})$ so that $\psi(v)-\iota\circ \phi(a_{v})+\int_{0}^{1}{\mathcal H}(dv)\in W^{\prime}_{l}(B^{1})$ is exact.
By $H^{1}({\mathcal M}^{\ast})={\rm ker}\, d_{\vert {\mathcal M}^{1}}$, such $a_{v}$ is unique.
In fact,  if $a_{v}^{\prime}$ is  another one, we have $d(a_{v}-a_{v}^{\prime})=0$ and $[\iota\circ \phi(a_{v}-a_{v}^{\prime})]$ is trivial in  $H^{1}(B^{\ast})$ and hence by $H^{1}(B^{\ast})\cong H^{1}({\mathcal M}^{\ast})={\rm ker} d_{\vert {\mathcal M}^{1}}$, we have $a_{v}-a_{v}^{\prime}=0$.
Let 
\[b_{v}=\delta_{C}\left(\psi(v)-\iota\circ \phi(a_{v})+\int_{0}^{1}{\mathcal H}(dv)\right)\in W^{\prime}_{l}(B^{0}).
\]
Then, define ${\mathcal I}(v)=a_{v}$ and ${\mathcal H}(v)=\psi(v)+\int^{t}_{0}{\mathcal H}(dv)-db_{v}\otimes  t-b_{v}\otimes dt$.
We obtain ${\mathcal I}$ and ${\mathcal H}$ on $\mathcal N_{i+1}^{\ast}=\mathcal N_{i}^{\ast}\otimes \bigwedge \V_{i+1}$ and so on $\mathcal N^{\ast}$ by induction.

By the construction, we can say that for any $x\in {\mathcal N}^{1}$, writing ${\mathcal H}(x)=\alpha+\beta\otimes dt$ with $\alpha \in B^{1}\otimes (t) $ and $\beta\in B^{0}\otimes (t)$, we have $\beta\in C$.

Consider  $\mathbb K$ as a  $\mathbb K$-DGA so that the component of degree $0$ is  $\mathbb K$ and the differential is trivial.
 An  {\em augmentation} of a  $\mathbb K$-DGA $A^{\ast}$ is a morphism $A^{\ast}\to \mathbb K$ which induces an isomorphism 
$H^{0}(A^{\ast})\cong \mathbb K$.
\begin{definition}\label{augMHD}
An  {\em augmented  $\mathbb K$-mixed-Hodge diagram} is  $({\mathcal D},\epsilon_{A},\epsilon_{B})$ so that 
\begin{itemize}
%\item ${\mathcal D}=\{(A^{*}, W_{*})\to^{\iota}(B^{*}, W_{*},F^{*})\}$ is  an $\mathbb K$-mixed-Hodge diagram  such that that the restriction $d_{\vert B^{0}}:B^{0}\to B^{1}$ of the differential on $B^{0}$ is strictly compatible with the filtrations $W_{\ast}$.
\item $\epsilon_{A} :A^{\ast}\to  \mathbb K$ and  $\epsilon_{B} :B^{\ast}\to  \C$ are augmentations.
\item $\epsilon_{B}\circ\iota=\epsilon_{A}$.

\end{itemize}

\end{definition}

Now we can define  Morgan mixed Hodge structures of augmented  $\mathbb K$-mixed-Hodge diagrams by the following way.

\begin{definition}
For an  augmented  $\mathbb K$-mixed-Hodge diagram  $({\mathcal D},\epsilon_{A},\epsilon_{B})$,
 we take the subspace $C={\rm ker}\epsilon_{B \vert B^{0}}$.
Consider  ${\mathcal I}$ and ${\mathcal H}$ constructed as above.
We call the $\mathbb K$-mixed-Hodge structure on $\mathcal M^{\ast}$ associated   with  ${\mathcal I}$ a {\em Morgan mixed Hodge structure of an augmented  $\mathbb K$-mixed-Hodge diagram}  $({\mathcal D},\epsilon_{A},\epsilon_{B})$.
\end{definition}

\subsection{Examples of augmented mixed Hodge diagrams}\label{EXMHD}

\begin{example}\label{kah}
Let $M$ be a compact complex manifold admitting  a K\"ahler metric.
On the de Rham complex $A^{\ast}(M)$ with the bigrading $A^{n}(M)\otimes \C=\bigoplus_{p+q=n} A^{p,q}(M)$,
define the increasing filtration $W_{\ast}$ on $A^{\ast}(M)$ so that $W_{-1}(A^{\ast}(M))=0$ and $W_{0}(A^{\ast}(M))=A^{\ast}(M)$.
For the usual Hodge filtration $F^{\ast}$ on $A^{\ast}(M)\otimes \C$ i.e. $F^{r}(A^{\ast}(M)\otimes \C)=\bigoplus_{p\ge r} A^{p,q}(M)$, we can easily check that 
\[\mathcal{D}(M)=\{(A^{*}(M), W_{*})\to^{\rm id}(A^{\ast}(M)\otimes \C, W_{*},F^{*})\}\]
 is a $\mathbb R$-mixed Hodge diagram.

Let $x\in M$.
Then we have the augmentation $\epsilon_{x}:A^{\ast}(M)\to \R$ so that $\epsilon_{x}(f)=f(x)$ for $f\in A^{0}(M)$.
$(\mathcal{D}(M), \epsilon_{x},\epsilon_{x})$ is an augmented $\mathbb R$-mixed Hodge diagram.

\end{example}

\begin{example}

Let $\overline{M}$ be a compact K\"ahler manifold and $D\subset \overline{M}$ a normal crossing divisor.
Consider the complement $M=\overline{M}-D$ called a quasi-K\"ahler manifold. 
We define the real analytic logarithmic complex ${\mathcal A}^{\ast}(M ,{\rm Log}\, D)$ on $M$ introduced by Navarro Aznar \cite{NA}.
Denote by ${\mathcal A}^{\ast}(Y)$ the DGA of real analytic differential forms on a real analytic manifold $Y$.
We define the sub-DGA ${\mathcal A}^{\ast}(M ,{\rm Log}( D))\subset {\mathcal A}^{\ast}(M)$ by the following way.
For any  $x\in M$, we have a neighborhood $B$ which admits an isomorphism  $B\cong \Delta^{n}$ so that $B\cap M\cong (\Delta^{\ast})^{l}\times \Delta^{n-l}$ where $\Delta$ is the unit disc and $\Delta^{\ast}=\Delta-\{0\}$.
For such neighborhood, an element in ${\mathcal A}^{\ast}(M ,{\rm Log}( D))$ is locally written by an element in the sub-${\mathcal A}^{\ast}(B)$-algebra ${\mathcal A}^{\ast}(B,{\rm Log}( D))$ of ${\mathcal A}^{\ast}(B\cap M)$ generated by
\[{\rm Re}\left(\frac{dz_{i}}{z_{i}}\right),\,\, {\rm Im}\left(\frac{dz_{i}}{z_{i}}\right),\,\, \lambda_{i}={\rm log}z_{i}\bar{z_{i}}, \,\, 1\le i \le l.
\]
On  ${\mathcal A}^{\ast}(B,{\rm Log}( D))$, we define the multiplicative weight filtration so that sections of ${\mathcal A}^{\ast}(B)$ are of weight $0$ and 
\[{\rm Re}\left(\frac{dz_{i}}{z_{i}}\right),\,\, {\rm Im}\left(\frac{dz_{i}}{z_{i}}\right),\,\, \lambda_{i}={\rm log}z_{i}\bar{z_{i}}, \,\, 1\le i \le l.
\]
are of weight $1$.
We define the filtration $W_{\ast}$ on ${\mathcal A}^{\ast}(M ,{\rm Log}(D))$ so that 
$\alpha\in W_{k}({\mathcal A}^{\ast}(M ,{\rm Log}(D))$ if on any $B$, $\alpha$ is presented by  forms of weight $\le k$ on $B$.

By the bigrading ${\mathcal A}^{\ast}(M ,{\rm Log}\, D)_{ \C}=\bigoplus {\mathcal A}^{p,q}(M ,{\rm Log}( D))$ associated with the complex structure on $M$, we define the decreasing filtration $F^{\ast}$ on ${\mathcal A}^{\ast}(M ,{\rm Log}(D))_{\C}$ by
\[F^{r}({\mathcal A}^{\ast}(M,{\rm log}( D))\otimes \C)=\bigoplus_{p\ge r} {\mathcal A}^{p,q}(M ,{\rm Log}(D)).
\]

Navarro Aznar proves that 
\[{\mathcal D}({\mathcal A}^{\ast}(M ,{\rm Log}(D)))=\{ ({\mathcal A}^{\ast}(M ,{\rm Log}(D)), W_{\ast})\to ^{\rm id} ({\mathcal A}^{\ast}(M ,{\rm Log}(D)), F^{\ast}, W_{\ast})\}\]
 is a $\R$-mixed Hodge diagram, the inclusion ${\mathcal A}^{\ast}(M ,{\rm Log}(D))\subset A^{\ast}(U)$ is a quasi-isomorphism and the $\R$-mixed Hodge structure on $H^{k}({\mathcal A}^{\ast}(M ,{\rm Log}(D)))$ associated with  ${\mathcal D}({\mathcal A}^{\ast}(M ,{\rm Log}(D)))$ is identified with the Deligne's $\R$-mixed Hodge structure on $H^{k}(M,\R)$ for any $k$.

Let $x\in M$.
We define the augmentation $\epsilon_{x}: {\mathcal A}^{\ast}(M ,{\rm Log}(D))\to  \R$ by
$\epsilon_{x}(f)=f(x)$ for $f\in {\mathcal A}^{0}(M ,{\rm Log}(D))$.
We have the  augmented $\mathbb R$-mixed Hodge diagram $({\mathcal D}({\mathcal A}^{\ast}(M ,{\rm Log}(D))), \epsilon_{x}, \epsilon_{x})$.

By \cite{Bu}, we can obtain ${\mathcal C}^{\infty}$-version.
Define ${\mathcal E}^{\ast}(M ,{\rm Log}(D)))$ by the sub-${\mathcal C}^{\infty}(\overline{M})$-module in $A^{\ast}(M)$ generated by ${\mathcal A}^{\ast}(M ,{\rm Log}(D))$.
We should remark that this is different from the tensor product ${\mathcal C}^{\infty}(\overline{M})\otimes_{ {\mathcal A}^{0}(\overline{M})}{\mathcal A}^{\ast}(M ,{\rm Log}(D)))$.
Define the filtration $W_{\ast}$ so that each $W_{k}({\mathcal E}^{\ast}(M ,{\rm Log}(D)))$ is generated by  $W_{k}({\mathcal A}^{\ast}(M ,{\rm Log}(D)))$.
Define  $F^{\ast}$ on ${\mathcal E}^{\ast}(M ,{\rm Log}(D)))_{\C}$ by the Hodge filtration.
Then by  Burgos shows that  the inclusion ${\mathcal A}^{\ast}(M ,{\rm Log}(D))\subset {\mathcal E}^{\ast}(M ,{\rm Log}(D))$ is a quasi-isomorphism.

We have the   $\mathbb R$-mixed Hodge diagram 
\[{\mathcal D}({\mathcal E}^{\ast}(M ,{\rm Log}(D)))=\{ ({\mathcal E}^{\ast}(M ,{\rm Log}(D)), W_{\ast})\to ^{\rm id} ({\mathcal E}^{\ast}(M ,{\rm Log}(D)), F^{\ast}, W_{\ast})\}.\]
For each $x\in M$,
$({\mathcal D}({\mathcal E}^{\ast}(M ,{\rm Log}(D))), \epsilon_{x}, \epsilon_{x})$ is an augmented $\mathbb R$-mixed Hodge diagram.

\end{example}

\section{The main equivalence}
\subsection{Flat connections of DGAs}
Let $A^{\ast}$ be a $\mathbb K$-DGA.
We define the category ${\mathcal F}(A^{\ast})$ so that: 
\begin{itemize}
\item Objects are $(V,\omega)$
\begin{itemize}
\item $V$ is a finite-dimensional $\mathbb K$-vector space.
\item $\omega \in A^{1}\otimes {\rm End}(V)$ and satisfies  the Maurer-Cartan equation 
\[d\omega+\omega\wedge\omega=0.
\]
\end{itemize}
\item For $(V_{1},\omega_{1}),(V_{2},\omega_{2})\in {\rm Ob}({\mathcal F}(A^{\ast}))$, morphisms from $(V_{1},\omega_{1})$ to $(V_{2},\omega_{2})$ are $a\in A^{0}\otimes  {\rm Hom}(V_{1}, V_{2})$ satisfying 
\[da+\omega_{2}a-a\omega_{1}=0.
\]
\end{itemize}

We also  define the category ${\mathcal F}^{Wnil}(A^{\ast})$ so that: 
\begin{itemize}
\item Objects are $(V, W_{\ast},\omega)$ so that
\begin{itemize}
\item $V$ is a finite-dimensional $\mathbb K$-vector space with an increasing filtration $W_{\ast}$.
\item $\omega \in A^{1}\otimes W_{-1}({\rm End}(V))$ and satisfies the Maurer-Cartan equation 
\[d\omega+\omega\wedge\omega=0.
\]
\end{itemize}
\item For $(V_{1}, W_{\ast},\omega_{1}),(V_{2}, W_{\ast},\omega_{2})\in {\rm Ob}({\mathcal F}^{Wnil}(A^{\ast}))$, morphisms from $(V_{1}, W_{\ast},\omega_{1})$ to $(V_{2}, W_{\ast},\omega_{2})$ are $a\in A^{0}\otimes  W_{0}({\rm Hom}(V_{1}, V_{2}))$ satisfying 
\[da+\omega_{2}a-a\omega_{1}=0.
\]
\end{itemize}
We define the full subcategory $F^{nil}(A^{\ast})$ of ${\mathcal F}(A^{\ast})$ such that 
\[{\rm Ob}({\mathcal F}^{nil}(A^{\ast}))=\{(V,\omega)\in  {\rm Ob}({\mathcal F}(A^{\ast}))\vert (V, W_{\ast},\omega)\in {\rm Ob}({\mathcal F}^{Wnil}(A^{\ast}))\}.\]

\begin{example}
Let $\g$ be a $\mathbb K$-Lie algebra.
We consider the $\mathbb K$-DGA $\bigwedge \g^{\ast}$.
Then if an element in $\g^{\ast}\otimes {\rm End}(V)$ satisfies the Maurer-Cartan equation, then it is identified with a  Lie algebra representation $\g \to {\rm End}(V)$.
Thus we can identify  ${\mathcal F}(\bigwedge \g^{\ast})$ with the category ${\rm Rep}(\g)$ of finite-dimensional representations of $\g$.
We can also identify  ${\mathcal F}^{Wnil}(\bigwedge \g^{\ast})$ with the category ${\rm Rep}^{Wnil}(\g)$ of filtered nilpotent representations of $\g$ i.e. objects are representations $\g\to W_{-1}({\rm End}(V))$ for a finite-dimensional filtered $\mathbb K$-vector space $(V,W_{\ast})$ and morphisms are linear maps which are compatible with filtrations and commute with representations. 
By this, we can identify ${\mathcal F}^{nil}(\bigwedge \g^{\ast})$ with the category ${\rm Rep}^{nil}(\g)$ of nilpotent representations of $\g$.
\end{example}

Let $A^{\ast}$ and $B^{\ast}$ be $\mathbb K$-DGAs and $\phi:A^{\ast}\to B^{\ast}$ a morphism of $\mathcal K$-DGAs.
Then $\phi$ induces the functor ${\mathcal F}^{Wnil}(\phi):{\mathcal F}^{Wnil}(A^{\ast})\to {\mathcal F}^{Wnil}(B^{\ast})$
and ${\mathcal F}^{nil}(\phi):{\mathcal F}^{nil}(A^{\ast})\to {\mathcal F}^{nil}(B^{\ast})$

\begin{proposition}\label{flaeq}
If  $\phi:A^{\ast}\to B^{\ast}$ is a $1$-quasi-isomorphism, then the functor ${\mathcal F}^{Wnil}(\phi):{\mathcal F}^{Wnil}(A^{\ast})\to {\mathcal F}^{Wnil}(B^{\ast})$ (resp. ${\mathcal F}^{nil}(\phi):{\mathcal F}^{nil}(A^{\ast})\to {\mathcal F}^{nil}(B^{\ast})$)  is an equivalence.

\end{proposition}

\begin{proof}
This proposition is given by the same arguments in \cite[Subection 10.1]{Kasuya}.
The  fully-faithfulness of ${\mathcal F}^{Wnil}(\phi)$ is proved by standard  arguments of homology algebra as in the proof of \cite[Proposition 10.1.1]{Kasuya}.
We can prove that ${\mathcal F}^{Wnil}(\phi)$ is   essentially-surjective by \cite[Lemma 10.1.2]{Kasuya}.
For our convenience, we refer the statement.

 We take a grading $V=\bigoplus V_{i}$ of vector space which is compatible with the filtration $W_{\ast}$.
 Corresponding to this grading, we have the grading ${\rm End}(V)=\bigoplus U_{i}$ and $U_{i}U_{j}\subset U_{i+j}$.
 For $\omega \in {\rm Ob}({\mathcal F}^{Wnil}(B^{\ast}))$, we write $\omega=\sum_{i\le -1} \omega _{i}$ with $\omega_{i}\in B^{1}\otimes U_{i}$. 
 By the  Maurer-Cartan equation, for each $k$, we have
 \[d\omega_{k}=-\sum_{i+j=k}\omega_{i}\wedge \omega_{j}.
 \] 
 We denote $a_{0}={\rm Id}\in {\rm End}(V)$.
\begin{lemma}[{\rm \cite[Lemma 10.1.2]{Kasuya}}]\label{eqmilll}

For all positive integers $k$,  there exist $\Omega_{k}\in A^{1}\otimes U_{k}$ and $a_{k} \in B^{0}\otimes  U_{k}$
such that 
\[d\Omega_{k}=-\sum _{i+j=k}\Omega_{i}\wedge \Omega_{j}
\qquad
{\rm and} \qquad
da_{k}=\sum_{i+j=k}(-\omega_{i}a_{j}+ a_{i}\phi(\Omega_{j})).
\]
\end{lemma} 
\begin{remark}\label{LEmr}
This is proved by the induction and the inductive step is direct consequence of a $1$-quasi-isomorphism $\phi:A^{\ast}\to B^{\ast}$.
If $A^{\ast}$ and $B^{\ast}$ are filtered and $d:A^{1}\to A^{2}$ and $d:B^{0}\to B^{1}$ are strictly compatible, then we can take all $\Omega_{k}$ and $a_{k}$ in the same filtration degree as  $\omega$.
For a   subspace $C\subset B^{0}$ so that the restriction $d_{\vert C}: C\to d(B^{0})$ of the differential $d$ is bijective,
we can take   $a_{k}\in C\otimes  U_{k}$.
\end{remark}
Let $\Omega=\sum\Omega_{i}\in A^{1} \otimes W_{-1}({\rm End}(V))$ and  $a=\sum a_{i}\in {\rm Id}+B^{0}\otimes W_{-1}({\rm End}(V))$.
We obtain the equations
\[d\Omega+\Omega\wedge \Omega=0
\qquad
{\rm and}
\qquad da+\omega a-a\phi(\Omega)=0.
\]
Then $(V, W_{\ast},\phi(\Omega))$ is isomorphic to the given $(V, W_{\ast},\omega)$ via $a$.

\end{proof}

 Let ${\mathcal M}^{\ast}$ be the $1$-minimal model of a $\mathbb K$-DGA $A^{\ast}$ with a $1$-quasi-isomorphism $\phi:{\mathcal M}^{\ast}\to  A^{\ast}$.
By this proposition, we have the equivalence ${\mathcal F}^{Wnil}(\phi):{\mathcal F}^{Wnil}({\mathcal M}^{\ast})\to {\mathcal F}^{Wnil}(A^{\ast})$ of categories.
Consider the dual Lie algebra $\frak n$ of the minimal DGA ${\mathcal M}^{\ast}$.
We can say that ${\mathcal F}^{Wnil}(A^{\ast})$  is equivalent to ${\rm Rep}^{W}(\frak n)$.

Let  $\epsilon_{A}:A^{\ast}\to \mathbb K$ be  a augmentation of a $\mathbb K$-DGA $A^{\ast}$.
For any $(V, W_{\ast},\omega)\in {\rm Obj}({\mathcal F}^{Wnil}(A^{\ast}))$, by iterating the use  of the cohomology long exact sequence and the five-Lemma, we can easily say that the map $\epsilon :H^{0}_{\omega}(A^{\ast}\otimes V)= {\rm ker } (d+\omega)_{\vert  A^{0}\otimes V}\to V$ is injective.
This implies an important consequence.
We define the functor $\varepsilon={\mathcal F}^{Wnil}(\epsilon):  {\mathcal F}^{Wnil}(A^{\ast})\to  {\mathcal F}^{Wnil}(\mathbb K)$.
Then this is a faithful functor.

\begin{proposition}[{\rm cf.\cite{Man}, \cite[Proposition 7.2.3]{Kasuya}}]\label{gauge}
Let $A^{\ast}$ be a $\mathbb K$-DGA and $\epsilon_{A}:A^{\ast}\to \mathbb K$ an augmentation.
For  $(V, W_{\ast},\omega_{0}), (V, W_{\ast},\omega_{1})\in  Ob({\mathcal F}^{Wnil}(A^{\ast}))$,
we assume that we have $(\tilde\omega, V, W_{\ast})\in  Ob({\mathcal F}^{Wnil}(A^{\ast}\otimes (t,dt)))$ such that:
\begin{itemize}
\item $\tilde\omega_{\vert t=0}=\omega_{0}$ and $\tilde\omega_{\vert t=1}=\omega_{1}$;
\item   $\tilde\omega=\alpha+\beta dt$ such that $\alpha\in A^{1}\otimes (t)\otimes {\rm End} (V)$ and $\beta\in C\otimes (t)\otimes {\rm End} (V)$ where $C={\rm ker}\epsilon_{A \vert A^{0}}$.
\end{itemize}
Then, there exists a unique  $a={\rm Id}_{V} + A$ such that $A\in C\otimes W_{-1}({\rm End}(V))$  and 
$\omega_{1}=a\omega_{0}a^{-1}-daa^{-1} $ (i.e. $(V, W_{\ast},\omega_{0})$ is isomorphic to $(V, W_{\ast},\omega_{1})$ via $a$).

\end{proposition}
\begin{proof}

We use the techniques in  \cite[Lemma 5.6, Proposition 5.7 and the proof of Theorem 5.5]{Man}.
We will find $A(t)\in C\otimes (t)\otimes  W_{-1}({\rm End}(V))$
such that $\alpha(t)=\tilde{a}\omega_{0}\tilde{a}^{-1}-d_{A^{\ast}}\tilde{a}\tilde{a}^{-1}$ for $\tilde{a}=\exp(A(t))$.
As  explained around  \cite[Proposition 5.7]{Man}, it is sufficient to solve the differential equation
\[\frac{d}{dt}A(t)+\gamma^{A}(t)=\beta(t)
\]
with $A(0)=0$ for certain $\gamma^{A}(t) \in A^{0}\otimes (t)\otimes W_{-1}({\rm End}(V))$.
By using  the Baker-Campbell-Hausdorff formula $\exp (X)\exp(Y)=\exp(X+Y+\frac{1}{2}[X,Y]+\frac{1}{12}([X,[X,Y]]+[Y,[Y,X]])\dots)$ and the Taylor expansion $A(t+h)=A(t)+\frac{d}{dt}A(t)h+\dots$, we have
\[\exp\left(A(t+h)\right)\exp\left(-A(t)\right)=\exp\left((A^{\prime}(t)+\gamma^{A}(t))h+\delta(t,h)h^{2}\right).
\]

We take a decomposition $V=\bigoplus_{i} V_{i}$  so that  $\bigoplus_{i\le k}V_{i}=W_{k}((V))$.
This decomposition induces the decomposition ${\rm End}(V)=\bigoplus {\rm End}_{i}(V)$ such that $\bigoplus_{i\le k}\bigoplus {\rm End}_{i}(V)=W_{k}({\rm End}(V))$ and $[ {\rm End}_{i}(V), {\rm End}_{j}(V)]\subset  {\rm End}_{i+j}(V)$.
We write $A(t)=\sum_{i}A_{i}(t)$, $\gamma^{A}(t) =\sum_{i}\gamma^{A}_{i}(t)$ and $\beta(t)=\sum_{i}\beta_{i}(t)$ associated with this decomposition.
Then each $\gamma^{A}_{i}(t)$ is a linear combination of iterated products of $A_{j}(t)$ and $A^{\prime}_{j}(t)$ with $j>i$.
Thus, inductively, we can solve $A(t)=\sum_{i}A_{i}(t)$ as
\[A_{i}(t)=\int_{0}^{t}\left(\beta_{i}(t)-\gamma^{A}_{i}(t)\right)
\]
and we can say $A_{i}(t)\in C\otimes (t)\otimes  {\rm End}_{i}(V)$ by $\beta\in C\otimes (t)\otimes {\rm End} (V)$.

By the faithfulness of the functor $\varepsilon$, $a\in {\rm Hom}((V,W_{\ast}, \omega_{0}), (V,W_{\ast},\omega_{1}))$ with $\varepsilon (a)= {\rm Id}_{V}$  is unique.
Thus the uniqueness  follows.

\end{proof}

\subsection{The category $VMHS^{u}({\mathcal D})$}\label{VMH}
Let ${\mathcal D}=\{(A^{*}, W_{*})\to^{\iota}(B^{*}, W_{*},F^{*})\}$ be an $\mathbb K$-mixed-Hodge diagram.
We define the category $VMHS^{u}({\mathcal D})$ such that:
\begin{itemize}
\item Objects are $(V, W_{\ast}, F^{\ast}, \omega,\omega^{\prime}, a)$ so that:
\begin{itemize}
\item $(W_{\ast}, F^{\ast})$ is a $\mathbb K$-mixed Hodge structure on a finite-dimensional $\mathbb K$-vector space $V$
\item $\omega\in {\rm Ob}({\mathcal F}^{Wnil}(A^{\ast}))$.
\item  $\omega^{\prime}\in {\rm Ob}({\mathcal F}^{Wnil}(B^{\ast}))$ satisfying \[\omega^{\prime} \in F^{0}\left(B^{\ast}\otimes W_{-1}({\rm End}(V_{\C}))\right).\]
\item $a\in {\rm Hom}((V_{\C}, W_{\ast}, \omega^{\prime}), (V_{\C}, W_{\ast},\iota(\omega)))$ so that $a\in {\rm Id}_{V}+B^{0}\otimes W_{-1}({\rm End}(V_{\C})$.

\end{itemize}
\item   Morphisms from $(V_{1}, W_{\ast} ,F^{\ast}, \omega_{1},\omega_{1}^{\prime}, a_{1})$ to  $(V_{2}, W_{\ast} ,F^{\ast}, \omega_{2},\omega_{2}^{\prime}, a_{2})$ are $(b,b^{\prime})$ 
so that
\begin{itemize}
\item $b\in {\rm Hom}((V_{1},W_{\ast}, \omega_{1}), (V_{2},W_{\ast},\omega_{2}))$ in ${\mathcal F}^{Wnil}(A^{\ast})$.
\item $b^{\prime}\in {\rm Hom}((V_{1\C}, W_{\ast}, \omega^{\prime}_{1}),(V_{2\C}, W_{\ast},\omega^{\prime}_{2}))$ satisfying 
\[
b^{\prime}\in F^{0}\left(B^{0} \otimes  W_{0}({\rm Hom}(V_{1\C}, V_{2\C}))\right) .\]
\item The diagram 
 \[\xymatrix{
(V_{1\C},W_{\ast},\iota( \omega_{1}))\ar[r]^{\iota(b)}&(V_{2\C},W_{\ast}, \iota(\omega_{2}))\\
(V_{1\C},W_{\ast}, \omega_{1}^{\prime})\ar[r]^{b^{\prime}}\ar[u]^{a_{1}}&(V_{2\C},W_{\ast}, \omega_{2}^{\prime})\ar[u]^{a_{2}}
}
\]
commutes.
\end{itemize}
\end{itemize}

In case ${\mathcal D}$ is the $\R$-mixed Hodge diagram on a compact K\"ahler manifold $M$ as in Example \ref{kah},
this category is identified with the category of real variations of mixed Hodge structures over $M$ (see \cite{Kasuya}). 

\subsection{The main equivalence}\label{eqeq2}
Let $({\mathcal D},\epsilon_{A},\epsilon_{B})$ be an augmented  $\mathbb K$-mixed-Hodge diagram.
We take the subspace $C={\rm ker}\epsilon_{B \vert B^{0}}$ as the above condition.
We fix a $1$-quasi-isomorphism $\phi:{\mathcal M}^{\ast}\to A^{\ast}$  which is compatible with the filtrations and a diagram 
$\xymatrix{
B^{\ast}&\ar[l]_{\iota}A^{\ast}_{\C}\ar[r]^{\bar\iota}&\overline{B^{\ast}}\\
&{\mathcal N}^{\ast}\ar[lu]^{\psi} \ar[u]^{\psi^{\prime\prime}}\ar[ru]_{\psi^{\prime}}
}$ 
as in Theorem \ref{Morgan}.  

We consider the  canonical ${\mathcal I}$ and ${\mathcal H}$ as in Subsection \ref{IHH}.
By theorem \ref{Morgan}, we have the $\mathbb K$-mixed Hodge structure on ${\mathcal M}^{\ast}$.
Consider the dual Lie algebra $\frak n$ of the minimal DGA ${\mathcal M}^{\ast}$.
Then we have the dual $\mathbb K$-mixed Hodge structure $(W_{\ast}, F^{\ast})$ on $\frak n$.

We consider the category ${\rm Rep}({\frak n}, W_{\ast},F^{\ast})$ such that:
\begin{itemize}
\item Objects are $(\Omega, V, W_{\ast},F^{\ast})$
\begin{itemize}
\item $(W_{\ast},F^{\ast})$ is a $\mathbb K$-mixed Hodge structure on a finite-dimensional $\mathbb K$-vector space $V$.
\item $\Omega: {\frak n}\to  {\rm End}(V)$ is a Lie algebra representation which is a morphism of $\mathbb K$-mixed Hodge structures.
\end{itemize} 
\item Morphisms are morphisms of $\frak n$-modules which are also morphisms of $\mathbb K$-mixed Hodge structures.

\end{itemize}
By $\frak n=W_{-1}(\frak n)$, we can say that  ${\rm Rep}({\frak n}, W_{\ast},F^{\ast})$ can be seen as a subcategory of ${\rm Rep}^{Wnil}(\frak n)$.

We construct objects in $VMHS^{u}({\mathcal D})$ from objects in ${\rm Rep}({\frak n}, W_{\ast},F^{\ast})$.
\begin{construction}\label{acons}
Let \[(\Omega, V, W_{\ast},F^{\ast})\in {\rm  Ob} ({\rm Rep}({\frak n}, W_{\ast},F^{\ast})).\]
Considering $\Omega$ as an element in  ${\mathcal M}^{1}\otimes {\rm End}(V)$, we have \[(\Omega, V, W_{\ast})\in {\rm Ob}({\mathcal F}^{Wnil}({\mathcal M}^{\ast})).\]
Thus we have $(\phi(\Omega), V, W_{\ast})\in {\rm Ob}({\mathcal F}^{Wnil}(A^{\ast}))$ and $(\psi\circ {\mathcal I}^{-1}(\Omega), V_{\C}, W_{\ast})\in {\rm Ob}({\mathcal F}^{Wnil}(B^{\ast}))$.
Since $\Omega$ is a morphism of $\mathbb K$-mixed Hodge structure, we have $\Omega \in F^{0}({\mathcal M}^{1}_{\C}\otimes {\rm End}(V_{\C}))$ which says $\Omega\in \sum_{P+r\ge 0} {\mathcal I}((N^{1})^{P,Q}\otimes F^{r}(V_{\C})$.
By the property of $\psi : {\mathcal N}^{\ast}\to  B^{\ast}$, we have $\psi\circ {\mathcal I}^{-1}(\Omega)\in F^{0}\left(B^{\ast}\otimes W_{-1}({\rm End}(V_{\C}))\right)$.
We consider ${\mathcal H}\circ {\mathcal I}^{-1}(\Omega)\in B^{\ast}\otimes (t ,dt)$. 
Then we have ${\mathcal H}\circ {\mathcal I}^{-1}(\Omega)_{\vert t=0}=\psi \circ {\mathcal I}^{-1}(\Omega)$ and ${\mathcal  H}\circ {\mathcal I}^{-1}(\Omega)_{\vert t=1}=\iota\circ \phi (\Omega)$.
By Proposition \ref{gauge}, we have $A\in C\otimes W_{-1}({\rm End}(V))$ such that denoting $a={\rm Id}_{V} + A$ we have
$\iota\circ \phi (\Omega)=a\psi \circ {\mathcal I}^{-1}(\Omega)a^{-1}-da a^{-1} $.
Thus, we obtain $(V, W_{\ast}, F^{\ast}, \phi (\Omega), \psi \circ {\mathcal I}^{-1}(\Omega), a)\in {\rm Ob}(VMHS^{u}({\mathcal D}) )$.
\end{construction}

\begin{definition}
We define the functor $\Phi_{C}:{\rm Rep}({\frak n}, W_{\ast},F^{\ast})  \to  VMHS^{u}({\mathcal D}) $ such that 
for $(\Omega, V, W_{\ast},F^{\ast})\in {\rm  Ob} ({\rm Rep}({\frak n}, W_{\ast},F^{\ast}))$,
\[\Phi_{C}((\Omega, V, W_{\ast},F^{\ast}))=(V, W_{\ast}, F^{\ast}, \phi (\Omega), \psi \circ {\mathcal I}^{-1}(\Omega), a)\in {\rm Ob}(VMHS^{u})
\]
and  for  $f\in {\rm Hom}\left((\Omega_{1}, V_{1}, W_{\ast},F^{\ast}), (\Omega_{2}, V_{2}, W_{\ast},F^{\ast})\right)$, 
$\Phi_{C}(f)=(f, f)$
where    $a\in {\rm Hom}((V_{\C}, W_{\ast}, \psi \circ {\mathcal I}^{-1}(\Omega)), (V_{\C}, W_{\ast},\iota\circ \phi (\Omega)))$  is constructed as in Construction \ref{acons}.

\end{definition}

We check that $\Phi_{C}(f)=(f, f)$ is actually a morphism in $VMHS^{u}({\mathcal D}) $.
Consider the isomorphisms $a(1): (V_{1\C},W_{\ast},\psi \circ {\mathcal I}^{-1}(\Omega_{1}))\cong (V_{1\C},W_{\ast},\iota\circ \phi( \Omega_{1}))$ and  $a(2): (V_{2\C},W_{\ast},\psi \circ {\mathcal I}^{-1}(\Omega_{2}))\cong (V_{1\C},W_{\ast},\iota\circ \phi( \Omega_{2}))$  as in  Proposition  \ref{gauge},  the diagram 
 \[\xymatrix{
(V_{1\C},W_{\ast},\iota\circ \phi( \Omega_{1}))\ar[r]^{f}&(V_{2\C},W_{\ast}, \iota\circ \phi( \Omega_{2}))\\
(V_{1\C},W_{\ast},\psi \circ {\mathcal I}^{-1}(\Omega_{1}))\ar[r]^{f}\ar[u]^{a_{1}}&(V_{2\C},W_{\ast}, \psi \circ {\mathcal I}^{-1}(\Omega_{2}))\ar[u]^{a_{2}}
}
\]
commutes by the following reason.
By $fa_{1}, a_{2}f\in {\rm Hom}((V_{1\C},W_{\ast},\psi \circ {\mathcal I}^{-1}(\Omega_{1}),(V_{2\C},W_{\ast}, \iota\circ \phi( \Omega_{2}))) $ with $\varepsilon(fa_{1})=f=\varepsilon(a_{2}f)$, we have $fa_{1}=a_{2}f$.
Hence, $\Phi_{C}:{\rm Rep}({\frak n}, W_{\ast},F^{\ast})  \to  VMHS^{u}({\mathcal D}) $ is indeed  a functor.

\begin{theorem}\label{EQVMH}
The functor $\Phi_{C}:{\rm Rep}({\frak n}, W_{\ast},F^{\ast})  \to  VMHS^{u}({\mathcal D}) $ is an equivalence.
We can take a quasi-inverse $\Phi_{C}^{\prime}:VMHS^{u}({\mathcal D})\to {\rm Rep}({\frak n}, W_{\ast},F^{\ast})   $ of $\Phi_{C}:{\rm Rep}({\frak n}, W_{\ast},F^{\ast})  \to  VMHS^{u}({\mathcal D}) $ so that
for any $(V, W_{\ast}, F^{\ast}, \omega,\omega^{\prime}, a)\in {\rm Ob}(VMHS^{u}({\mathcal D}))$, we can write
\[\Phi_{C}^{\prime}(V, W_{\ast}, F^{\ast}, \omega,\omega^{\prime}, a)=( \Omega, V, W_{\ast}, \epsilon_{B}(a)F^{\ast}).
\]

\end{theorem}
\begin{proof}
For $(\Omega_{1}, V_{1}, W_{\ast},F^{\ast}), (\Omega_{2}, V_{2}, W_{\ast},F^{\ast})\in {\rm Ob}({\rm Rep}({\frak n}, W_{\ast},F^{\ast}))$, by Proposition \ref{flaeq},  we obtain the  identifications
\[{\rm Hom}((\Omega_{1}, V_{1}, W_{\ast}),  (\Omega_{2}, V_{1}, W_{\ast}))= {\rm Hom}((\phi(\Omega_{1}), V_{1}, W_{\ast}),  (\phi(\Omega_{2}), V_{1}, W_{\ast}))\]
 and 
 \[{\rm Hom}((\Omega_{1}, V_{1\C}, W_{\ast}),  (\Omega_{2}, V_{1\C}, W_{\ast}))= {\rm Hom}((\psi\circ {\mathcal I}^{-1}(\Omega_{1}), V_{1\C}, W_{\ast}),  (\psi\circ {\mathcal I}^{-1}(\Omega_{2}), V_{1\C}, W_{\ast})).\]
By these identifications, we can easily show the fully-faithfulness of $\Phi_{C}$.

We prove that the functor is essentially-surjective.
Let $(V, W_{\ast}, F^{\ast}, \omega,\omega^{\prime}, a)\in {\rm Ob}(VMHS^{u}({\mathcal D}))$.
Then, by Proposition \ref{flaeq}, we can take
 \[(\Omega, V, W_{\ast})\in {\rm Ob}({\mathcal F}^{Wnil}({\mathcal M}^{\ast}))\qquad {\rm and}\qquad (\Omega^{\prime}, V, W_{\ast})\in {\rm Ob}({\mathcal F}^{Wnil}({\mathcal N}^{\ast}))\]
  and isomorphisms 
\[b:(\phi(\Omega), V, W_{\ast})\cong (\omega, V, W_{\ast})
\qquad {\rm and}\qquad b^{\prime}: (\psi(\Omega^{\prime}), V, W_{\ast})\cong (\omega^{\prime}, V, W_{\ast}).
\]
We notice that we can take $\Omega^{\prime}\in F^{0}({\mathcal N}^{1}\otimes W_{-1}({\rm End}(V)))$ and $b^{\prime}\in F^{0}\left(B^{0} \otimes  W_{0}({\rm Hom}(V_{1\C}, V_{2\C}))\right)$.
Indeed since the differentials $d$ on ${\mathcal N}^{\ast}$ and $d$ on $B^{0}$  are strictly compatible with the filtration $F^{\ast}$,  on  constructions as in  Lemma \ref{eqmilll}, we can take  each term in  $F^{0}$.
By Proposition \ref{gauge}, we have an isomorphism $\hat{a}: (V_{\C},W_{\ast}, \iota\circ \phi(\Omega))\cong (V_{\C},W_{\ast}, \psi\circ {\mathcal I}^{-1}(\Omega))$.

Now we have the diagram 
 \[\xymatrix{
(V_{\C},W_{\ast},\iota( \omega))&\ar[l]^{\iota(b)}(V_{\C},W_{\ast},\iota\circ \phi(\Omega))\\
& (V_{\C},W_{\ast}, \psi\circ {\mathcal I}^{-1}(\Omega)) \ar[u]^{\hat{a}}\\
(V_{\C},W_{\ast}, \omega^{\prime})\ar[uu]^{a}&\ar[l]^{b^{\prime}}(V_{\C},W_{\ast}, \psi( \Omega^{\prime}))\\
}
\]
Thus we have the isomorphism $c:(V_{\C},W_{\ast}, \psi\circ {\mathcal I}^{-1}(\Omega))\cong  (V_{\C},W_{\ast}, \psi( \Omega^{\prime}))$ associated with this diagram.
By Proposition \ref{flaeq} we can regard $c$ as an isomorphism $c: (V_{\C},W_{\ast}, {\mathcal I}^{-1}(\Omega))\cong (V_{\C},W_{\ast},  \Omega^{\prime})$.
Finally, $(V, W_{\ast}, cF^{\ast}, \Omega) \in {\rm Ob}( {\rm Rep}({\frak n}, W_{\ast},F^{\ast}) )$ and 
\[\Phi_{C}((\Omega, V, W_{\ast}, cF^{\ast}))=(V, W_{\ast}, cF^{\ast},  \phi(\Omega), \psi\circ {\mathcal I}^{-1}(\Omega), \hat{a})\]
is isomorphic to $(V, W_{\ast}, F^{\ast}, \omega,\omega^{\prime}, a)\in {\rm Ob}(VMHS^{u}({\mathcal D}))$.

We prove the second assertion.
 for an  augmented  $\mathbb K$-mixed-Hodge diagram  $({\mathcal D},\epsilon_{A},\epsilon_{B})$, in the diagram as above, we can take $b$ and $b^{\prime}$ such that  $\epsilon_{A}(b)=Id_{V}$ and $\epsilon_{B}(b^{\prime})=Id_{V_{\C}}$ as Remark \ref{LEmr}.
By Proposition \ref{gauge}, $\hat{a}\in Id_{V_{\C}}+ C\otimes W_{-1}({\rm End}(V_{\C}))$ and so  $\epsilon_{B}(\hat{a})=Id_{V_{\C}}$.
Hence we can have $c=\epsilon_{B}(a)$ and taking the quasi-inverse $\Phi_{C}^{\prime}:VMHS^{u}({\mathcal D})\to {\rm Rep}({\frak n}, W_{\ast},F^{\ast})   $ of $\Phi_{C}:{\rm Rep}({\frak n}, W_{\ast},F^{\ast})  \to  VMHS^{u}({\mathcal D}) $ as  $\Phi_{C}^{\prime}(V, W_{\ast}, F^{\ast}, \omega,\omega^{\prime}, a)=(\Omega, V, W_{\ast}, cF^{\ast})$, the second assertion follows.

\end{proof}

\section{Non-abelian mixed Hodge structures and the main statement}\label{NMHsect}
\subsection{Tannakian categories}
A category $\mathcal C$ with a functor $\otimes: {\mathcal C}\times {\mathcal C}\to {\mathcal C}$  is a {\em (additive)  ${\mathbb K}$-tensor category} if:
\begin{itemize}
\item  $\mathcal C$ is an additive $\C$-linear category.
\item $\otimes : {\mathcal C}\times {\mathcal C}\to {\mathcal C}$ is a bi-linear functor which satisfies the associativity and commutativity. (see \cite{DM})
\item There exists an identity object $({\bf 1}, u)$ that is ${\bf 1}\in {\rm Ob}({\mathcal C})$ with an isomorphism $u:{\bf 1}\to {\bf 1}\otimes {\bf 1}$ satisfying the functor ${\mathcal C}\ni V\mapsto {\bf 1}\otimes V\in {\mathcal C}$ is an equivalence of categories.

\end{itemize}
A $\mathbb K$-tensor category $\mathcal C$ is {\em rigid}  if all objects admit duals.

For two $\mathbb K$-tensor categories   $({\mathcal C}_{1},\otimes_{1})$ and $({\mathcal C}_{2},\otimes_{2})$,
a {\em tensor functor} is a functor $F:{\mathcal C}_{1}\to {\mathcal C}_{2}$ with a functorial isomorphism $c_{U,V}:F(U)\otimes F(V)\to F(U\otimes V)$ so that $(F,c)$ is compatible with the associativities and commutativities of $({\mathcal C}_{1},\otimes_{1})$ and $({\mathcal C}_{2},\otimes_{2})$ (see \cite{DM}) and 
for an identity object  $({\bf 1}, u)$ of $({\mathcal C}_{1},\otimes_{1})$ $(F({\bf 1}), F(u))$ is an identity object of $({\mathcal C}_{2},\otimes_{2})$.

The category ${\rm Vect}_{\C}$ of $\C$-vector space with the usual tensor product $\otimes$ is a tensor category.
For a tensor category $({\mathcal C},\otimes)$, an exact faithful tensor functor ${\mathcal C}\to {\rm Vect}_{\C}$ is called a {\em fiber functor} for ${\mathcal C}$.
A {\em neutral $\mathbb K$-Tannakian category}  is an abelian rigid   $\mathbb K$-tensor category $\mathcal C$ with a fiber functor $\omega: {\mathcal C}\to {\rm Vect}_{\mathbb K}$  such that   $\mathbb K= {\rm End}(\bf 1)$.

\begin{theorem}[\cite{DM}]
Every neutral  $\mathbb K$-Tannakian category $\mathcal C$ is equivalent to the category ${\rm Rep}(G)$ of finite-dimensional representations of an pro-algebraic group $G$ over $\mathbb K$.
More precisely,  this correspondence $\mathcal C\mapsto G$ is  a contravariant  functor $\Pi$ from the category of neutral  Tannakian categories to the category of pro-algebraic groups over $\mathbb K$
where morphisms of neutral  Tannakian categories are exact faithful tensor functors commuting with fiber functors.
\end{theorem}
We call $G=\Pi(\mathcal C)$ the Tannakian dual of a neutral  $\mathbb K$-Tannakian category $\mathcal C$.

\begin{example}
Consider the category ${\mathcal MHS}_{\mathbb K}$ of $\mathbb K$-mixed Hodge structures.
Then, for \[(V_{1},W_{\ast}, F^{\ast}),   (V_{2},W_{\ast}, F^{\ast})\in  {\rm Ob}({\mathcal MHS}_{\mathbb K}),\] 
we can define the $\mathbb K$-mixed Hodge structure on 
the tensor product  $V_{1}\otimes V_{2}$ of $\mathbb K$-vector spaces by the standard way.
Thus, ${\mathcal MHS}_{\mathbb K}$ is a $\mathbb K$-tensor category.
Since ${\mathcal MHS}_{\mathbb K}$ is abelian, we can easily say that ${\mathcal MHS}_{\mathbb K}$ is a neutral $\mathbb K$-Tannakian category with the fiber functor ${\rm Ob}({\mathcal MHS}_{\mathbb K})\ni (V,W_{\ast}, F^{\ast})\mapsto V\in  {\rm Ob}({\rm Vect}_{\mathbb K})$.

\end{example}

\begin{example}
Let $\g$ be a $\mathbb K$-Lie algebra.
Then the category ${\rm Rep}(\g)$ (resp. ${\rm Rep}^{nil}(\g)$ with the usual tensor product and the natural functor 
${\rm Rep}(\g)\to {\rm Vect}_{\mathbb K}$ (resp. ${\rm Rep}^{nil}(\g)\to {\rm Vect}_{\mathbb K}$) is a neutral $\mathbb K$-Tannakian category.

It is known that if $\frak n$ is pro-nilpotent, then the Tannakian dual $U=\Pi({\rm Rep}^{nil}(\frak n))$ is a pro-unipotent group  over $\mathbb K$ such that the Lie algebra of $U$ is $\frak n$.
\end{example}

\subsection{Non-abelian mixed Hodge structures}\label{NMHsec}
The following definition is inspired by Arapura's work in \cite{Ara}.
\begin{definition}
A {\em $\mathbb K$-non-abelian mixed Hodge structure} ($\mathbb K$-NMHS) is  an abelian rigid   $\mathbb K$-tensor category $\mathcal C$ with exact faithful tensor functors $\tau_{1}: \mathcal C\to  {\mathcal MHS}_{\mathbb K}$ and $\tau_{2}:   {\mathcal MHS}_{\mathbb K}\to \mathcal C$ satisfying $\tau_{1}\circ \tau_{2}={\rm id}_{ {\mathcal MHS}_{\mathbb K}}$.

\end{definition}
With the composition of $\tau_{1}$ and the fiber functor ${\rm Ob}({\mathcal MHS}_{\mathbb K})\ni (V,W_{\ast}, F^{\ast})\mapsto V\in  {\rm Ob}({\rm Vect}_{\mathbb K})$, $\mathcal C$ is a neutral  $\mathbb K$-Tannakian category.
Considering the Tannakian duals $\Pi(\mathcal C)$ and $\Pi({\mathcal MHS}_{\mathbb K})$, by the morphisms $\Pi(\tau_{1})$ and $\Pi(\tau_{2})$, we have the splitting $\Pi(\mathcal C)=\Pi({\mathcal MHS}_{\mathbb K})\ltimes G$ where 
 $G={\rm ker}\Pi(\tau_{2})$.
 Thus a $\mathbb K$-non-abelian mixed Hodge structure  can be considered as  a pro-algebraic group $G$ over $\mathbb K$
with an algebraic action of $\Pi({\mathcal MHS}_{\mathbb K})$ as defined in \cite{Ara}.

Let   $({\mathcal D},\epsilon_{A},\epsilon_{B})$ be an augmented  $\mathbb K$-mixed-Hodge diagram.
Consider the category  $VMHS^{u}({\mathcal D})$ as in Section \ref{VMH}.
Naturally, we consider $VMHS^{u}({\mathcal D})$ as a tensor category.
We consider the natural faithful functor $\kappa:  {\mathcal MHS}_{\mathbb K}\to VMHS^{u}({\mathcal D})$ so that $\kappa (V, W_{\ast}, F^{\ast})= (V, W_{\ast}, F^{\ast},0,0,{\rm id}_{V})\in  {\rm Ob}(VMHS^{u}({\mathcal D}))$.
Define the functor $\varepsilon :VMHS^{u}({\mathcal D})\to {\mathcal MHS}_{\mathbb K}$ by
\[\varepsilon (V, W_{\ast}, F^{\ast}, \omega,\omega^{\prime}, a)=(V, W_{\ast}, \epsilon_{B}(a)F^{\ast}).
\]
Then $\varepsilon$ is a tensor functor and  $\varepsilon\circ \kappa ={\rm id}_{ {\mathcal MHS}}$.
Thus we can say that $(VMHS^{u}({\mathcal D}), \varepsilon, \kappa)$ is a $\mathbb K$-NMHS

\subsection{Functoriality and Ho-morphisms}
A morphism between $\mathbb K$-mixed-Hodge  diagrams ${\mathcal D}_{1}=\{(A^{*}_{1}, W_{*})\to^{\iota_{1}}(B^{*}_{1}, W_{*},F^{*})\}$ and ${\mathcal D}_{2}=\{(A^{*}_{2}, W_{*})\to^{\iota_{2}}(B^{*}_{2}, W_{*},F^{*})\}$ is 
$(f,g, H)$ such that
\begin{itemize}
\item $f:(A_{1}^{\ast},W_{\ast})\to (A_{2}^{\ast},W_{\ast})$ is a morphism of filtered  $\mathbb K$-DGAs.
\item $g:(B_{1}^{\ast},W_{\ast},F^{\ast})\to (B_{2}^{\ast},W_{\ast},F^{\ast})$  is a morphism of bifiltered  $\mathbb C$-DGAs.
\item the equation $\iota_{2}\circ f=g\circ \iota_{1}$  holds.
\end{itemize}
A morphism between augmented  $\mathbb K$-mixed-Hodge  diagrams $({\mathcal D}_{1},\epsilon_{A_{1}},\epsilon_{B_{1}})$ and $({\mathcal D}_{2},\epsilon_{A_{2}},\epsilon_{B_{2}})$ is a morphism 
$(f,g)$ between ${\mathcal D}_{1}$ and ${\mathcal D}_{2}$ so that $ \epsilon_{A_{2}}\circ f=\epsilon_{A_{1}}$,  $ \epsilon_{B_{2}}\circ g=\epsilon_{B_{1}}$ and  $ \epsilon_{B_{2}}\otimes {\rm id}_{(t,dt)}\circ H=\epsilon_{A_{1}}$.
Corresponding to a morphism $(f,g)$ between augmented  $\mathbb K$-mixed-Hodge  diagrams $({\mathcal D}_{1},\epsilon_{A_{1}},\epsilon_{B_{1}})$ and $({\mathcal D}_{2},\epsilon_{A_{2}},\epsilon_{B_{2}})$, we functorially obtain the tensor functor $T_{(f,g)}:VMHS^{u}({\mathcal D}_{1})\to VMHS^{u}({\mathcal D}_{2})$ given by 
\[{\rm Ob}(VMHS^{u}({\mathcal D}_{1}))\ni 
(V, W_{\ast}, F^{\ast}, \omega, \omega^{\prime}, a)\mapsto
(V, W_{\ast}, F^{\ast}, f(\omega), g(\omega^{\prime}), g(a))\in {\rm Ob}(VMHS^{u}({\mathcal D}_{2}))
\]
and the maps
 $(b, b^{\prime})\mapsto (f(b), g(b^{\prime}))$ between  morphisms such that  the diagram
\[\xymatrix{
VMHS^{u}({\mathcal D}_{1})\ar@<1.0 mm>[rr]^{\varepsilon_{1}}\ar[d]^{T_{(f,g)}}&&\ar@<1.0mm>[ll]^{\kappa_{1}}{\mathcal MHS}_{\mathbb K}\ar[d]^{=}\\
 VMHS^{u}({\mathcal D}_{2})\ar@<1.0 mm>[rr]^{\varepsilon_{2}}&&\ar@<1.0 mm>[ll]^{\kappa_{2}}{\mathcal MHS}_{\mathbb K}
}
\]
commutes.

A ho-morphism between $\mathbb K$-mixed-Hodge  diagrams ${\mathcal D}_{1}=\{(A^{*}_{1}, W_{*})\to^{\iota_{1}}(B^{*}_{1}, W_{*},F^{*})\}$ and ${\mathcal D}_{2}=\{(A^{*}_{2}, W_{*})\to^{\iota_{2}}(B^{*}_{2}, W_{*},F^{*})\}$ is 
$(f,g, H)$ such that
\begin{itemize}
\item $f:(A_{1}^{\ast},W_{\ast})\to (A_{2}^{\ast},W_{\ast})$ is a morphism of filtered  $\mathbb K$-DGAs.
\item $g:(B_{1}^{\ast},W_{\ast},F^{\ast})\to (B_{2}^{\ast},W_{\ast},F^{\ast})$  is a morphism of bifiltered  $\mathbb C$-DGAs.
\item $H:A_{1\C}^{\ast}\to B_{2}\otimes (t,dt)$ is a  homotopy from $\iota_{2}\circ f$ to $g\circ \iota_{1}$  which is compatible with the filtrations $W_{\ast}$.
\end{itemize}
\begin{remark}
In \cite{Mor}, Morgan originally considers  ho-morphisms as morphisms of  $\mathbb K$-mixed-Hodge  diagrams.
But Ho-morphisms can not be composed in general.
Hence, they do not define a category.
Nevertheless, they are very important see \cite{Cir}.
\end{remark}
We extend the construction of the tensor functor $T_{(f,g)}$ to ho-morphisms.
Since Ho-morphisms can not be composed, this construction is not functorial unlike morphisms.
But the construction is interesting and may be useful for further studies.

A ho-morphism between augmented  $\mathbb K$-mixed-Hodge  diagrams $({\mathcal D}_{1},\epsilon_{A_{1}},\epsilon_{B_{1}})$ and $({\mathcal D}_{2},\epsilon_{A_{2}},\epsilon_{B_{2}})$ is a ho-morphism 
$(f,g, H)$ between ${\mathcal D}_{1}$ and ${\mathcal D}_{2}$ so that $ \epsilon_{A_{2}}\circ f=\epsilon_{A_{1}}$,  $ \epsilon_{B_{2}}\circ g=\epsilon_{B_{1}}$ and  $ \epsilon_{B_{2}}\otimes {\rm id}_{(t,dt)}\circ H=\epsilon_{A_{1}}$.
For $(V, W_{\ast}, F^{\ast}, \omega, \omega^{\prime}, a)\in {\rm Ob}(VMHS^{u}({\mathcal D}_{1}))$,
  applying   Proposition \ref{gauge} for $\tilde\omega  =H(\omega)$, we have a canonical  $\check{a} \in Id_{V_{\C}}+ {\rm ker}\epsilon_{B_{2}}\otimes W_{-1}({\rm End}(V_{\C}))$ such that $g\circ \iota_{1}(\omega)=\check{a}\iota_{2}\circ f(\omega) \check{a}^{-1}-d\check{a}\check{a}^{-1}$.
We have
\[d(\check{a}^{-1}g(a))=\check{a}^{-1}g(a)g(\omega^{\prime})-\iota_{2}\circ f(\omega)\check{a}^{-1}g(a).
\]
Thus, we have $(V, W_{\ast}, F^{\ast}, f(\omega), g(\omega^{\prime}), \check{a}^{-1}g(a))\in {\rm Ob}(VMHS^{u}({\mathcal D}_{2}))$.
By $\epsilon \circ g\circ \iota_{1}= \epsilon \circ  \iota_{2}\circ f$,
 corresponding to a ho-morphism $(f,g, H)$ between augmented  $\mathbb K$-mixed-Hodge  diagrams $({\mathcal D}_{1},\epsilon_{A_{1}},\epsilon_{B_{1}})$ and $({\mathcal D}_{2},\epsilon_{A_{2}},\epsilon_{B_{2}})$, we have the tensor functor $T_{(f,g, H)}:VMHS^{u}({\mathcal D}_{1})\to VMHS^{u}({\mathcal D}_{2})$ given by 
\[{\rm Ob}(VMHS^{u}({\mathcal D}_{1}))\ni 
(V, W_{\ast}, F^{\ast}, \omega, \omega^{\prime}, a)\mapsto
(V, W_{\ast}, F^{\ast}, f(\omega), g(\omega^{\prime}), \check{a}^{-1}g(a))\in {\rm Ob}(VMHS^{u}({\mathcal D}_{2}))
\]
and the maps
 $(b, b^{\prime})\mapsto (f(b), g(b^{\prime}))$ between  morphisms such that  the diagram
\[\xymatrix{
VMHS^{u}({\mathcal D}_{1})\ar@<1.0 mm>[rr]^{\varepsilon_{1}}\ar[d]^{T_{(f,g, H)}}&&\ar@<1.0mm>[ll]^{\kappa_{1}}{\mathcal MHS}_{\mathbb K}\ar[d]^{=}\\
 VMHS^{u}({\mathcal D}_{2})\ar@<1.0 mm>[rr]^{\varepsilon_{2}}&&\ar@<1.0 mm>[ll]^{\kappa_{2}}{\mathcal MHS}_{\mathbb K}
}
\]
commutes.
For a  morphism $(b,b^{\prime})$ from $(V_{1}, W_{\ast} ,F^{\ast}, \omega_{1},\omega_{1}^{\prime}, a_{1})$ to  $(V_{2}, W_{\ast} ,F^{\ast}, \omega_{2},\omega_{2}^{\prime}, a_{2})$ in $VMHS^{u}({\mathcal D}_{1})$,
we should check $(f(b),g(b^{\prime}))$ is a morphism 
\[{\rm from}\qquad (V_{1}, W_{\ast} ,F^{\ast}, f(\omega_{1}), g(\omega_{1}^{\prime}), \check{a}^{-1}_{1}g(a_{1}))\qquad {\rm
 to}\qquad(V_{2}, W_{\ast} ,F^{\ast}, f(\omega_{2}), g(\omega_{2}^{\prime}), \check{a}^{-1}_{2}g(a_{2}))\] in $VMHS^{u}({\mathcal D}_{2})$.
 So we check  the diagram 
 \[\xymatrix{
(V_{1\C},W_{\ast},\iota( f(\omega_{1})))\ar[r]^{\iota(f(b))}&(V_{2\C},W_{\ast}, \iota(f(\omega_{2})))\\
(V_{1\C},W_{\ast}, g(\omega_{1}^{\prime}))\ar[r]^{g(b^{\prime})}\ar[u]^{\check{a}^{-1}_{1}g(a_{1})}&(V_{2\C},W_{\ast}, g(\omega_{2}^{\prime}))\ar[u]^{\check{a}^{-1}_{2}g(a_{2})}
}
\]
commutes.
By $\check{a}_{1}, \check{a}_{2} \in Id_{V_{\C}}+ {\rm ker}\epsilon_{B_{2}}\otimes W_{-1}({\rm End}(V_{\C}))$, we have
\begin{multline*}
\epsilon_{B_{2}}(\iota(f(b))\check{a}^{-1}_{1}g(a_{1}))=\epsilon_{B_{2}}(\iota(f(b))g(a_{1}))\\
=\epsilon_{A_{1}}(b)\epsilon_{B_{1}}(a_{1})=\epsilon_{B_{1}}(\iota(b)a_{1})=\epsilon_{B_{1}}(a_{2}b^{\prime})\\
=\epsilon_{B_{2}}(g(a_{2})g(b^{\prime}))=\epsilon_{B_{2}}(g(a_{2}) \check{a}_{2}^{-1}g(b^{\prime})).
\end{multline*}
This implies $\iota(f(b))\check{a}^{-1}_{1}g(a_{1})=g(a_{2}) \check{a}_{2}^{-1}g(b^{\prime})$.
Hence $T_{(f,g, H)}:VMHS^{u}({\mathcal D}_{1})\to VMHS^{u}({\mathcal D}_{2})$ is indeed a functor.

A ho-morphism $(f,g, H)$ between augmented $\mathbb K$-mixed-Hodge  diagrams is quasi-isomorphism (resp. $1$-quasi-isomorphism) if $f$ and $g$ are quasi-isomorphisms (resp. $1$-quasi-isomorphisms).
The following proposition is easily proved by the essentially same argument in the proof of Theorem \ref{EQVMH}.
\begin{proposition}
If a ho-morphism $(f,g, H)$ between augmented  $\mathbb K$-mixed-Hodge  diagrams  ${\mathcal D}_{1}=\{(A^{*}_{1}, W_{*})\to^{\iota_{1}}(B^{*}_{1}, W_{*},F^{*})\}$ and ${\mathcal D}_{2}=\{(A^{*}_{2}, W_{*})\to^{\iota_{2}}(B^{*}_{2}, W_{*},F^{*})\}$ is a $1$-quasi-isomorphism, then  the tensor functor $T_{(f,g, H)}:VMHS^{u}({\mathcal D}_{1})\to VMHS^{u}({\mathcal D}_{2})$
is an equivalence and we can take a quasi-inverse $T_{(f,g, H)}^{\prime}:VMHS^{u}({\mathcal D}_{2})\to VMHS^{u}({\mathcal D}_{1})$ such that 
 the diagram
\[\xymatrix{
VMHS^{u}({\mathcal D}_{1})\ar@<1.0 mm>[rr]^{\varepsilon_{1}}\ar[d]^{T_{(f,g, H)}}&&\ar@<1.0mm>[ll]^{\kappa_{1}}{\mathcal MHS}_{\mathbb K}\ar[d]^{=}\\
 VMHS^{u}({\mathcal D}_{2})\ar@<1.0 mm>[rr]^{\varepsilon_{2}}\ar@<2.0mm>[u]^{T^{\prime}_{(f,g, H)}}&&\ar@<1.0 mm>[ll]^{\kappa_{2}}{\mathcal MHS}_{\mathbb K}
}
\]
commutes.
\end{proposition}

\begin{example}
Let $M=\overline{M}-D$ and $M^{\prime}=\overline{M}^{\prime}-D^{\prime}$ be two quasi-K\"ahler manifolds  and
 $f: \overline{M}\to \overline{M}^{\prime}$ be a holomorphic map so that $f(M)\subset M^{\prime}$.
We have the pull-back $f^{\ast}:{\mathcal A}^{\ast}(M^{\prime} ,{\rm Log}(D^{\prime}))\to {\mathcal A}^{\ast}(M ,{\rm Log}(D))$(resp.  $f^{\ast}:{\mathcal E}^{\ast}(M^{\prime} ,{\rm Log}(D^{\prime}))\to {\mathcal E}^{\ast}(M ,{\rm Log}(D))$) and hence we have the  morphism $(f^{\ast},f^{\ast})$ between $\mathbb R$-mixed Hodge diagrams ${\mathcal D}({\mathcal A}^{\ast}(M ,{\rm Log}(D)))$ and ${\mathcal D}({\mathcal A}^{\ast}(M^{\prime} ,{\rm Log}(D^{\prime})))$ (resp ${\mathcal D}({\mathcal E}^{\ast}(M ,{\rm Log}(D)))$ and ${\mathcal D}({\mathcal E}^{\ast}(M^{\prime} ,{\rm Log}(D^{\prime})))$.

For $x\in M$ and  $x^{\prime}\in M^{\prime}$, if $f(x)=x^{\prime}$, then we have the morphism between augmented $\mathbb R$-mixed Hodge diagrams 
\[({\mathcal D}({\mathcal A}^{\ast}(M ,{\rm Log}(D)),\epsilon_{x},\epsilon_{x}) ) \qquad ({\rm resp.} ({\mathcal D}({\mathcal E}^{\ast}(M ,{\rm Log}(D)),\epsilon_{x},\epsilon_{x}) )\]  and
 \[\left({\mathcal D}({\mathcal A}^{\ast}(M^{\prime} ,{\rm Log}(D^{\prime})), \epsilon_{x^{\prime}},\epsilon_{x^{\prime}}\right)\qquad ({\rm resp.} \left({\mathcal D}({\mathcal E}^{\ast}(M^{\prime} ,{\rm Log}(D^{\prime})), \epsilon_{x^{\prime}},\epsilon_{x^{\prime}}\right)).\]

Since the inclusion ${\mathcal A}^{\ast}(M ,{\rm Log}(D))\subset {\mathcal E}^{\ast}(M ,{\rm Log}(D))$ is a quasi-isomorphism, we have the quasi-isomorphism between augmented $\mathbb R$-mixed Hodge diagrams $({\mathcal D}({\mathcal A}^{\ast}(M ,{\rm Log}(D)),\epsilon_{x},\epsilon_{x})$ and $({\mathcal D}({\mathcal E}^{\ast}(M ,{\rm Log}(D)),\epsilon_{x},\epsilon_{x})$.

\end{example}

\subsection{The main statement}
Let   $({\mathcal D},\epsilon_{A},\epsilon_{B})$ be an augmented  $\mathbb K$-mixed-Hodge diagram.
We consider the  ${\rm Rep}({\frak n}, W_{\ast},F^{\ast}) $ as in Section \ref{eqeq2} associated with a Morgan mixed Hodge structure of an augmented  $\mathbb K$-mixed-Hodge diagram  $({\mathcal D},\epsilon_{A},\epsilon_{B})$ as in Section \ref{IHH}.
We define the functors $\tau_{1}: {\rm Rep}({\frak n}, W_{\ast},F^{\ast})\to  {\mathcal MHS}_{\mathbb K}$ and $\tau_{2}:   {\mathcal MHS}_{\mathbb K}\to {\rm Rep}({\frak n}, W_{\ast},F^{\ast})$ such that 
\[\tau_{1}: {\rm Ob}({\rm Rep}({\frak n}, W_{\ast},F^{\ast}))\ni (V, W_{\ast}, F^{\ast}, \Omega)\mapsto (V, W_{\ast}, F^{\ast})\in  {\rm Ob}({\mathcal MHS}_{\mathbb K})\]
 and  \[\tau_{2}: {\rm Ob}({\mathcal MHS}_{\mathbb K})\ni (V, W_{\ast}, F^{\ast})\to (V, W_{\ast}, F^{\ast},0)\in {\rm Ob}({\rm Rep}({\frak n}, W_{\ast},F^{\ast})).\]
 Since every morphism in ${\rm Rep}({\frak n}, W_{\ast},F^{\ast})$ is a morphism of  $\mathbb K$-mixed-Hodge structures,
 $\tau_{1}$ is indeed  a functor.
 Since  morphisms from $(V_{1}, W_{\ast}, F^{\ast},0)$ to $(V_{2}, W_{\ast}, F^{\ast},0)$ in ${\rm Rep}({\frak n}, W_{\ast},F^{\ast})$ are just  morphisms of  $\mathbb K$-mixed-Hodge structures,  $\tau_{2}$ is indeed  a functor.
 Then ${\rm Rep}({\frak n}, W_{\ast},F^{\ast}) $ with these functors is  a $\mathbb K$-NMHS.
Thus by the Theorem \ref{EQVMH}, we obtain the following statement.
\begin{theorem}\label{NMth}
Let   $({\mathcal D},\epsilon_{A},\epsilon_{B})$ be an augmented  $\mathbb K$-mixed-Hodge diagram.
Then the $\mathbb K$-NMHS $({\rm Rep}({\frak n}, W_{\ast},F^{\ast}),\tau_{1},\tau_{2}) $ associated with a Morgan mixed Hodge structure of an augmented  $\mathbb K$-mixed-Hodge diagram  $({\mathcal D},\epsilon_{A},\epsilon_{B})$   is equivalent to 
the $\mathbb K$-NMHS $(VMHS^{u}({\mathcal D}), \varepsilon, \kappa)$. 

\end{theorem}

\begin{remark}\label{funn}
It is   difficult to extend morphisms between augmented  $\mathbb K$-mixed-Hodge  diagrams to morphisms between systems of Morgan mixed Hodge structures as in Theorem \ref{Morgan}. 
Thus we can not precisely consider Morgan mixed Hodge structures of augmented  $\mathbb K$-mixed-Hodge  diagrams as functorial invariants but by Theorem \ref{NMth} we can consider them as representatives of functorial invariants $(VMHS^{u}({\mathcal D}), \varepsilon, \kappa)$ of  $\mathbb K$-mixed-Hodge  diagrams.

\end{remark}

\section{On compact K\"ahler manifolds}\label{cpkah}
We study the canonical $\mathbb R$-mixed-Hodge  diagrams associated with compact K\"ahler manifolds  and their $\R$-NMHS in detail.
We see the relation between   Morgan mixed Hodge structures of  augmented mixed Hodge diagrams on compact K\"ahler manifolds  and the mixed Hodge structures fundamental groups of compact K\"ahler manifolds as in \cite{HainI}.
Moreover,  we see that  we can explicitly construct  Morgan mixed Hodge structures as functorial invariants of K\"ahler manifolds with points in contrast to Remark \ref{funn}.

\subsection{Real Variations of mixed Hodge structures}\label{vmfun}
Let $M$ be a complex manifold (not necessarily compact or K\"ahler). 
A {\em real variation of mixed Hodge structure} ($\R$-VMHS)  over  $M$ is  $({\bf E}, {\bf W}_{\ast}, {\bf F}^{\ast})$ so that:
\begin{enumerate}
\item ${\bf E}$ is a local system of  finite-dimensional $\R$-vector spaces.
\item ${\bf W}_{\ast}$ is an increasing filtration of the local system ${\bf E}$.
\item  ${\bf F}^{\ast}$ is a decreasing filtration of the holomorphic vector bundle ${\bf E}\otimes_{\R}{\mathcal O}_{M}$.
\item The Griffiths transversality $D{\bf F}^{r}\subset A^{1}(M, {\bf F}^{r-1})$ holds where $D$ is the flat connection associated with the local system ${\bf E}_{ \C}$.
\item  For any $x\in M$, the fiber $({\bf E}_{x}, {\bf W}_{\ast x}, {\bf F}^{\ast}_{x})$ at $x$ is a $\R$-mixed-hodge structure.
\end{enumerate}
We call $({\bf E}, {\bf W}_{\ast}, {\bf F}^{\ast})$  unipotent if  the variation  on each $Gr_{k}^{\bf W}(\bf E) $ is constant (i.e. the local system is trivial and the holomorphic sub-bundles of $Gr_{k}^{\bf W}(\bf E\otimes_{\R}{\mathcal O}_{M}) $ associated with $\bf F$ are trivial bundles).

We consider the category $VMHS_{\R}(M)$ of $\R$-VMHSs over $M$ such that morphisms from $({\bf E}_{1}, {\bf W}_{\ast}, {\bf F}^{\ast})$ to $({\bf E}_{2}, {\bf W}_{\ast}, {\bf F}^{\ast})$ are  flat sections in  which are valued in ${\bf W}_{0}({\rm Hom}({\bf E}_{1}, {\bf E}_{2}))$  and 
${\bf F}^{0}({\rm Hom}({\bf E}_{1}, {\bf E}_{2})_{\C})$.
We define the full-subcategory $VMHS^{u}_{\R}(M)$ so that objects are unipotent $\R$-VMHSs.
For each $x\in M$,  define the functor $\epsilon_{x}: VMHS^{u}_{\R}(M)\to {\mathcal MHS}_{\mathbb R}$ as the taking fiber at $x$
\[\epsilon_{x}:{\rm Ob}(VMHS^{u}_{\R}(M))\ni ({\bf E}, {\bf W}_{\ast}, {\bf F}^{\ast})\mapsto ({\bf E}_{x}, {\bf W}_{\ast x}, {\bf F}_{x}^{\ast})\in {\rm Ob}({\mathcal MHS}_{\mathbb R}).
\]
We also define the functor  $\kappa^{\prime}: {\mathcal MHS}_{\mathbb R}\to VMHS^{u}_{\R}(M)$  such that for
$(V,W_{\ast},F^{\ast}) \in  {\rm Ob}({\mathcal MHS}_{\mathbb R})$, $\kappa^{\prime}((V,W_{\ast},F^{\ast}))$ is the trivial flat bundle $M\times V$ with the filtrations $M\times W_{\ast}$ and  $M\times F^{\ast}$.

On the de Rham complex $A^{\ast}(M)$ with the bigrading $A^{n}(M)\otimes \C=\bigoplus_{p+q=n} A^{p,q}(M)$,
consider    the Hodge filtration $F^{\ast}$ on $A^{\ast}(M)\otimes \C$ i.e. $F^{r}(A^{\ast}(M)\otimes \C)=\bigoplus_{p\ge r} A^{p,q}(M)$.
We also define the category $\widetilde{VMHS}^{u}_{\R}(M)$ 
such that:
\begin{itemize}
\item Objects are $(V, W_{\ast}, F^{\ast}, \omega,\omega^{\prime}, a)$ so that:
\begin{itemize}
\item $(W_{\ast}, F^{\ast})$ is a $\mathbb R$-mixed Hodge structure on a finite-dimensional $\mathbb R$-vector space $V$
\item $\omega\in {\rm Ob}({\mathcal F}^{Wnil}(A^{\ast}(M)))$.
\item  $\omega^{\prime}\in {\rm Ob}({\mathcal F}^{Wnil}(A^{\ast}(M)\otimes \C))$ satisfying \[\omega^{\prime} \in F^{0}\left(A^{\ast}(M)\otimes \C\otimes W_{-1}({\rm End}(V_{\C}))\right).\]
\item $a\in {\rm Hom}((V_{\C}, W_{\ast}, \omega^{\prime})), (V_{\C}, W_{\ast},\omega)$ so that $a\in {\rm Id}_{V}+A^{0}(M)\otimes \C\otimes W_{-1}({\rm End}(V_{\C})$.

\end{itemize}
\item   Morphisms from $(V_{1}, W_{\ast} ,F^{\ast}, \omega_{1},\omega_{1}^{\prime}, a_{1})$ to  $(V_{2}, W_{\ast} ,F^{\ast}, \omega_{2},\omega_{2}^{\prime}, a_{2})$ are $(b,b^{\prime})$ 
so that
\begin{itemize}
\item $b\in {\rm Hom}((V_{1},W_{\ast}, \omega_{1}), (V_{2},W_{\ast},\omega_{2})$ in ${\mathcal F}^{Wnil}(A^{\ast}(M))$.
\item $b^{\prime}\in {\rm Hom}((V_{1\C}, W_{\ast}, \omega^{\prime}_{1}),(V_{2\C}, W_{\ast},\omega^{\prime}_{2})$ satisfying 
\[
b^{\prime}\in F^{0}\left(A^{\ast}(M)\otimes \C \otimes  W_{0}({\rm Hom}(V_{1\C}, V_{2\C}))\right) .\]
\item The diagram 
 \[\xymatrix{
(V_{1\C},W_{\ast}, \omega_{1})\ar[r]^{b}&(V_{2\C},W_{\ast}, \omega_{2})\\
(V_{1\C},W_{\ast}, \omega_{1}^{\prime})\ar[r]^{b^{\prime}}\ar[u]^{a_{1}}&(V_{2\C},W_{\ast}, \omega_{2}^{\prime})\ar[u]^{a_{2}}
}
\]
commutes.
\end{itemize}
\end{itemize}
Then each $(V, W_{\ast}, F^{\ast}, \omega,\omega^{\prime}, a)\in {\rm Ob}(\widetilde{VMHS}^{u}_{\R}(M)) $ corresponds to the $\R$-VMHS $({\bf E}, {\bf W}_{\ast}, {\bf F}^{\ast})\in {\rm Ob}(VMHS^{u}_{\R}(M))$
such that ${\bf E}$ is the trivial $C^{\infty}$-vector bundle $M\times V$ with the flat connection $d+\omega$, ${\bf W}_{\ast}$ is defined by  ${\bf W}_{i}=M\times   W_{i}(V)$ and ${\bf F}^{\ast}$ is defined by  ${\bf F}^{r}=a(M\times F^{r}(V_{\C}))$ (see \cite[Section 7]{Kasuya}).
By the definition, this correspondence $\Upsilon:  {\rm Ob}(\widetilde{VMHS}^{u}_{\R}(M)\to {\rm Ob}(VMHS^{u}_{\R}(M)$ is a fully-faithful functor.
Taking global frames, every $({\bf E}, {\bf W}_{\ast}, {\bf F}^{\ast})\in {\rm Ob}(VMHS^{u}_{\R}(M))$ is isomorphic to $\Upsilon(V, W_{\ast}, F^{\ast}, \omega,\omega^{\prime}, a)$ for some $(V, W_{\ast}, F^{\ast}, \omega,\omega^{\prime}, a)\in {\rm Ob}(\widetilde{VMHS}^{u}_{\R}(M)) $.
Thus, the functor $\Upsilon:  \widetilde{VMHS}^{u}_{\R}(M))\to VMHS^{u}_{\R}(M))$  is an equivalence.
For  $(V, W_{\ast}, F^{\ast}, \omega,\omega^{\prime}, a)\in {\rm Ob}(\widetilde{VMHS}^{u}_{\R}(M)) $,  we have 
\[\epsilon_{x}\circ \Upsilon(V, W_{\ast}, F^{\ast}, \omega,\omega^{\prime}, a)= (V, W_{\ast}, a(x)F^{\ast})
\]
where $a(x)\in {\rm Aut}_{1}(V_{\C}, W_{\ast})$ is the value at $x$.

\subsection{Mixed Hodge diagrams over compact K\"ahler manifolds}\label{KAMD}

Let $M$ be a compact complex manifold admitting  a K\"ahler metric $g$.
 We consider the $\R$-mixed-Hodge diagram
\[\mathcal{D}(M)=\{(A^{*}(M), W_{*})\to^{\rm id}(A^{\ast}(M)\otimes \C, W_{*},F^{*})\}\]
 as in Example \ref{kah}.
We notice that $\widetilde{VMHS}^{u}_{\R}(M)=VMHS^{u}(\mathcal{D}(M))$.

Let $x\in M$.
Then we have the augmentation $\epsilon_{x}:A^{\ast}(M)\to \R$ so that $\epsilon_{x}(f)=f(x)$ for $f\in A^{0}(M)$.
We consider the augmented $\mathbb R$-mixed Hodge diagram $(\mathcal{D}(M), \epsilon_{x},\epsilon_{x})$.

By Theorem \ref{NMth} and Subsection \ref{vmfun}, we have:
\begin{theorem}\label{ppptNH}

The $\R$-NMHS $({\rm Rep}({\frak n}, W_{\ast},F^{\ast}),\tau_{1},\tau_{2}) $ associated with a Morgan mixed Hodge structure of an augmented  $\R$-mixed-Hodge diagram  $(\mathcal{D}(M), \epsilon_{x},\epsilon_{x})$ 
 is equivalent to  the $\R$-NMHS $(VMHS^{u}_{\R}(M), \epsilon_{x}, \kappa^{\prime})$.
\end{theorem}

For the purpose explained in Introduction, we construct an explicit filtered $1$-minimal model and bigraded $1$-minimal model.

\subsection{Canonical $1$-minimal models (\cite{Kasuya, KAB})}
Let $M$ be a compact complex manifold admitting  a K\"ahler metric $g$.
On the de Rham complex $A^{\ast}(M)$ with the bigrading $A^{n}(M)\otimes \C=\bigoplus_{p+q=n} A^{p,q}(M)$, we 
consider the differential operators $d$, $d^{c}$, $\partial$ and $\bar{\partial}$.
%, their adjoints $d^{\ast}$, $d^{c^{\ast}}$, $\partial^{\ast}$ and $\bar{\partial}^{\ast}$ associated with  $g$, their Laplacian operators $\Delta$, $\Delta^{c}$, $\Delta^{\prime}$ and $\Delta^{\prime\prime}$ and their Green operators $G$, $G^{c}$, $G^{\prime}$ and $G^{\prime\prime}$.
%Then we have 
%\[\Delta=\Delta^{c}=2\Delta^{\prime}=2\Delta^{\prime\prime}.\]
%We define 
%\[F^{c}=d^{\ast}Gd^{c^{\ast}}G^{c} \qquad {\rm and } \qquad  F^{\prime}=-2\sqrt{-1}F^{c}.\]
Then we have the following relations (see \cite[Lemma 5.11]{DGMS}).
\begin{proposition}\label{DDc}
\begin{description}
\item[($dd^{c}$-Lemma)]
On each $A^{i}(M)$, we have
\[{\rm im}d\cap {\rm ker}d^{c}={\rm ker}d\cap {\rm im}d^{c}={\rm im}dd^{c}.
\]
\item[($\partial\bar\partial$-Lemma)]
On each $A^{p,q}(M)$,
we have
\[{\rm im}\partial\cap {\rm ker}\bar\partial={\rm ker}\partial\cap {\rm im}\bar\partial={\rm im}\partial\bar\partial.
\]
\end{description}
\end{proposition}
We notice that on the $1$-forms, we have
\[{\rm im}d\cap {\rm ker}d^{c}={\rm ker}d\cap {\rm im}d^{c}=0
\]
and 
\[{\rm im}\partial\cap {\rm ker}\bar\partial={\rm ker}\partial\cap {\rm im}\bar\partial=0.
\]
By this, by these relations on $2$-forms and $1$-forms, for a $2$-form $\alpha$ satisfying $\alpha\in {\rm im}d\cap {\rm ker}d^{c}$(resp. $\alpha\in{\rm ker}\partial\cap {\rm im}\bar\partial$), we have a function $f$ such that $\alpha=dd^{c}f$ (resp. $\alpha=\partial\bar\partial f$) and the $1$-form $d^{c}f$ (resp. $\partial f$) is uniquely determined by $\alpha$ that is to say
$f$ is determined up to an additive constant.
We will denote such $1$-form by $d^{c}F^{c}(\alpha)$ (resp.  $\partial F^{\prime}(\alpha)$).

We consider the sub-DGA ${\rm ker}d^{c}\subset A^{\ast}(M)$ (resp. ${\rm ker}\partial\subset A^{\ast}(M)\otimes \C$).
We regard the cohomologies $H_{d^{c}}^{\ast}(M)$ and $H_{\partial}^{\ast}(M)$ as DGAs with the  differential operators $d$.
By these two relations,
we have the following (see \cite[Section 6]{DGMS}).
\begin{theorem}
\begin{itemize}
\item The inclusions 
\[{\rm ker}d^{c} \subset A^{\ast}(M)  \qquad {\rm and }\qquad  {\rm ker}\partial\subset  A^{\ast}(M)\otimes\C\]
are quasi-isomorphisms.
\item The quotients 
\[{\rm ker}d^{c} \to H^{\ast}_{d^{c}}(M)  \qquad {\rm and }\qquad  {\rm ker}\partial\to  H^{\ast}_{\partial}(M)\]
are quasi-isomorphisms.
\item $d=0$ on $H_{d^{c}}^{\ast}(M)$ and $H_{\partial}^{\ast}(M)$ and hence $H_{d^{c}}^{\ast}(M)\cong H^{\ast}(M,\R)$ and $H_{\partial}^{\ast}(M)\cong H^{\ast}(M,\C)$. 
\end{itemize}

\end{theorem}
%By this theorem, we can say that two DGAs $A^{\ast}(M)$ and $H^{\ast}(M,\R)$ have the same $1$-minimal model.
%Thus, by the argument in Example \ref{MIKKK}, the $1$-minimal model of $A^{\ast}(M)$ admits a sprit $\R$-mixed Hodge structure.
%But, this $\R$-mixed Hodge structure is not  interesting, because it is determined by the $\R$-hodge structure on $H^{1}(M,\R)$ and the cup product $H^{1}(M,\R)\wedge H^{1}(M,\R)\to H^{2}(M,\R)$.

We first construct the canonical $1$-minimal models $\phi:{\mathcal M}^{\ast}\to A^{\ast}(M)$ and $\psi:{\mathcal N}^{\ast}\to A^{\ast}(M)\otimes\C$.

\begin{construction}\label{ddc}
We construct ${\mathcal M}_{i}^{\ast}=\bigwedge ( \V_{1}\oplus \dots \oplus \V_{i})$ and the map $\phi_{i}:{\mathcal M}_{i}^{\ast}\to A^{\ast}(M)$ by the following way.
\begin{enumerate}
\item ${\mathcal V}_{1}={\mathcal H}^{1}(M)={\rm ker}d\cap {\rm ker}d^{c}\cap A^{1}(M)$. Define $\phi_{1}:\bigwedge  \V_{1}\to A^{\ast}(M)$ so that $\phi$ is the natural injection ${\mathcal V}_{1}={\mathcal H}^{1}(M)\hookrightarrow A^{1}(M)$ on $\V_{1}$.
\item Define $\V_{2}$ as the kernel of the cup product
\[ \bigwedge^{2} \V_{1}\to H^{2}(M,\R).
\]
Define the DGA ${\mathcal M}^{\ast}_{2}=\bigwedge (\V_{1}\oplus \V_{2})$ with the differential $d$ so that $d$ is $0$ on $\V_{1}$ and $d$ on $\V_{2}$ is the natural inclusion $\V_{2}\hookrightarrow  \bigwedge^{2} \V_{1}$.
Extend the morphism $\phi_{2}:{\mathcal M}^{\ast}_{2}\to A^{\ast}(M)$ so that  $\phi_{2}(v)=d^{c}F^{c}(\phi_{1}(dv))$ for $v\in \V_{2}$.
\item 
For $i\ge 2$, consider the DGA ${\mathcal M}^{\ast}_{i}=\bigwedge(\V_{1}\oplus \V_{2}\oplus\dots\oplus \V_{i})$ with the homomorphism $\phi_{i}:{\mathcal M}^{\ast}_{i}\to A^{\ast}(M)$ we have constructed.
We have  $\phi_{i}(v)\in {\rm im} d^{c}$ for $v\in \V_{2}\oplus\dots\oplus \V_{i}$.

Let
\[\V_{i+1}={\rm ker} d_{ \vert \sum_{j+k=i+1} \V_{j}\wedge \V_{k}}.\]
 Define the extended DGA ${\mathcal M}^{\ast}_{i+1}={\mathcal M}^{\ast}_{i}\otimes \bigwedge \V_{i+1}$
so that the differential $d$ is defined on $\V_{i+1}$ as the natural inclusion $\V_{i+1}\hookrightarrow \sum_{j+k=i+1} \V_{j}\wedge \V_{k}$.
Extend the morphism $\phi_{i+i}:{\mathcal M}^{\ast}_{i+1}\to A^{\ast}(M)$ is defined by 
$\phi_{i+1}(v)=d^{c}F^{c}(\phi_{i}(dv))$ for $v\in \V_{i+1}$.
\end{enumerate}

We have the grading ${\mathcal M}^{\ast}=\bigoplus {\mathcal M}^{\ast}(n)$ which commutes with the multiplication and the differential operator 
so that ${\mathcal M}^{1}(i)=\V_{i}$.
Define the filtration  $W_{\ast}$ on  ${\mathcal M}^{\ast}$ by $W_{i}({\mathcal M}^{\ast})=\bigoplus_{n\le i} {\mathcal M}^{\ast}(n)$.
\end{construction}

\begin{construction}\label{ddb}
We construct ${\mathcal N}_{i}^{\ast}=\bigwedge ( \W_{1}\oplus \dots \oplus \W_{i})$ and the map $\psi_{i}:{\mathcal N}_{i}^{\ast}\to A^{\ast}(M)\otimes \C$ by the following way.
\begin{enumerate}
\item ${\mathcal W}_{1}={\mathcal H}^{1}(M)\otimes \C={\rm ker}\partial\cap {\rm ker}\bar\partial\cap A^{1}(M)\otimes\C$. Define $\psi_{1}:\bigwedge  \W_{1}\to A^{\ast}(M)\otimes \C$ so that $\psi_{1}$ is the natural injection ${\mathcal V}_{1}={\mathcal H}^{1}(M)\hookrightarrow A^{1}(M)\otimes\C$ on $\W_{1}$.
\item Define $\W_{2}$ as the kernel of the cup product
\[ \bigwedge^{2} \W_{1}\to H^{2}(M,\C).
\]
Define the DGA ${\mathcal N}^{\ast}_{2}=\bigwedge (\W_{1}\oplus \W_{2})$ with the differential $d$ so that $d$ is $0$ on $\W_{1}$ and $d$ on $\W_{2}$ is the natural inclusion $\W_{2}\hookrightarrow  \bigwedge^{2} \W_{1}$.
Extend the morphism $\psi_{2}:{\mathcal N}^{\ast}_{2}\to A^{\ast}(M)\otimes \C$ so that  $\psi_{2}(w)=\partial F^{\prime}(\psi_{1}(dw))$ for $w\in \W_{2}$.
\item 
For $i\ge 2$, consider the DGA ${\mathcal N}^{\ast}_{i}=\bigwedge(\W_{1}\oplus \W_{2}\oplus\dots\oplus \W_{i})$ with the homomorphism $\psi_{i}:{\mathcal N}^{\ast}_{i}\to A^{\ast}(M)\otimes\C$ we have constructed.
We have  $\psi_{i}(w)\in {\rm im} \partial$ for $w\in \W_{2}\oplus\dots\oplus \W_{i}$.

Let
\[\W_{i+1}={\rm ker} d_{ \vert \sum_{j+k=i+1} \W_{j}\wedge \W_{k}}.\]
 Define the extended DGA ${\mathcal N}^{\ast}_{i+1}={\mathcal N}^{\ast}_{i}\otimes \bigwedge \W_{i+1}$
so that the differential $d$ is defined on $\W_{i+1}$ as the natural inclusion $\W_{i+1}\hookrightarrow \sum_{j+k=i+1} \W_{j}\wedge \W_{k}$.
Extend the morphism $\psi_{i+i}:{\mathcal N}^{\ast}_{i+1}\to A^{\ast}(M)\otimes\C$ is defined by 
$\psi_{i+1}(w)=\partial F^{\prime}(\psi_{i}(dw))$ for $w\in \W_{i+1}$.
\end{enumerate}

We have the grading ${\mathcal N}^{\ast}=\bigoplus {\mathcal N}^{\ast}(n)$ which commutes with the multiplication and the differential operator 
so that ${\mathcal N}^{1}(i)=\W_{i}$.
For the decomposition ${\mathcal H}^{1}(M)\otimes \C={\mathcal H}^{1,0}(M)\oplus {\mathcal H}^{0,1}(M)$,
we define $\W^{1,0}={\mathcal H}^{1,0}(M)$, $\W^{0,1}={\mathcal H}^{0,1}(M)$ and 
\[\W^{P,Q}={\rm ker} d_{ \vert \sum_{s+u=P,t+v=Q} \W^{s,t}\wedge \W^{u,v}}\]
inductively.
Then we have $\W_{n}=\bigoplus_{P+Q=n} \W^{P,Q}$.
By this, we have the bigrading ${\mathcal N}^{\ast}=\bigoplus ({\mathcal N}^{\ast})^{P,Q}$ which commutes with the multiplication and the differential operator 
so that $({\mathcal N}^{1})^{P,Q}=\W^{P,Q}$.
\end{construction}

Consider the $\mathbb R$-mixed Hodge diagram
$\mathcal{D}(M)=\{(A^{*}(M), W_{*})\to^{\rm id}(A^{\ast}(M)\otimes \C, W_{*},F^{*})\}$ as the last subsection.
By the $1$-quasi-isomorphism $\phi: ({\mathcal M}^{\ast},W_{\ast})\to (A^{\ast}, W_{\ast})$, $ ({\mathcal M}^{\ast},W_{\ast})$ is the filtered $1$-minimal model of the filtered DGA $(A^{\ast}, W_{\ast})$.
By the diagram 
 \[\xymatrix{
A^{\ast}(M)\otimes \C&\ar[l]_{\rm id}A^{\ast}(M)\otimes \C\ar[r]^{\overline{\rm id}}&\overline{A^{\ast}(M)\otimes \C}\\
&{\mathcal N}^{\ast}\ar[lu]^{\psi} \ar[u]^{\psi}\ar[ru]_{\bar\psi}
},
\]
$\mathcal N=\bigoplus_{0\le P,Q}({\mathcal N}^{\ast})^{P,Q}$ is the bigraded $1$-minimal model of $\mathcal{D}(M)$.

Thus, constructing as in  Section \ref{IHH} for the augmented $\mathbb R$-mixed Hodge diagram $(\mathcal{D}(M), \epsilon_{x},\epsilon_{x})$, we obtain the  explicit  Morgan mixed hodge structure associated with a compact K\"ahler manifold $M$ with a point $x\in M$.

\subsection{Computations}
We consider an explicit object in  $VMHS^{u}({\mathcal D}(M))$.
Let $\omega_{1},\omega_{2}\in {\mathcal H}^{1,0}(M)$ such that $\omega_{1}\wedge\overline\omega_{2}$ is null-cohomologous.
Consider real vector spaces $V_{0}$, $V_{1}$ and $V_{2}$ such that 
$V_{0}=\langle v_{0,0}\rangle $, $V_{1\C}=\langle v_{1,0}, v_{0,1}\rangle $ with $\overline{v_{1,0}}=v_{0,1}$ and $V_{2}=\langle v_{1,1}\rangle $.
Consider the direct sum $V=V_{0}\oplus V_{1}\oplus V_{2}$ with the $\R$-split $\R$-mixed Hodge structure $(W_{\ast}, F^{\ast})$ associated with the bigrading $V_{\C}^{i,j}=\langle v_{i,j}\rangle $. 
Write $(V_{\C}^{i,j})^{\ast}=\langle w_{-i,-j}\rangle$.
Let $\alpha_{1}=\omega_{1}\otimes w_{-1,0}\otimes v_{0,0}+\overline\omega_{1}\otimes w_{0,-1}\otimes v_{0,0}  \in {\mathcal H}^{1}(M)\otimes {\rm Hom}(V_{1},V_{0})$ and  $\alpha_{2}=\overline\omega_{2}\otimes w_{-1,-1}\otimes v_{1,0}+\omega_{2}\otimes w_{-1,-1}\otimes v_{0,1} \in {\mathcal H}^{1}(M)\otimes {\rm Hom}(V_{2},V_{1})$.
By the $dd^{c}$-Lemma (Proposition \ref{DDc}), we have $\beta\in A^{0}(M)_{\C}$ such that 
\[\omega_{1}\wedge \overline\omega_{2}=dd^{c}\gamma.
\]
We notice that $\gamma$ is determined up to an additive constant.
Associated with a point   $x\in M$,
we  take $\beta$ satisfying $\gamma(x)=0$.
Let $\beta=(\gamma +\overline\gamma)\otimes w_{-1,-1}\otimes v_{0,0}$.
Let $\omega=\left(
\begin{array}{ccc}
0& \alpha_{1}&d^{c}\beta  \\
0&    0 &\alpha_{2}\\
0&0&0 
\end{array}
\right)$, $\omega^{\prime}=\left(
\begin{array}{ccc}
0& \alpha_{1}&-2\sqrt{-1}\partial\beta  \\
0&    0 &\alpha_{2}\\
0&0&0 
\end{array}
\right)$ and $a_{\beta}=\left(
\begin{array}{ccc}
1& 0&-\sqrt{-1}\beta  \\
0&    1 &0\\
0&0&1 
\end{array}
\right)$.
Then we can easily check \[ (V, W_{\ast} ,F^{\ast}, \omega,\omega^{\prime}, a_{\beta})\in {\rm Ob}(VMHS^{u}({\mathcal D}(M))).\]
Otherwise, by  Construction \ref{ddc}, we have $v\in {\mathcal V}_{2}\otimes {\rm Hom}(V_{2},V_{0})$ such that $\phi(v)=d^{c}\beta$.
Since $\mathcal I$ and $\mathcal H$ are the identity and the natural injection on ${\mathcal W}_{1}={\mathcal H}^{1}(M)\otimes \C$ respectively, we have \[\psi\circ {\mathcal I}^{-1}(v)=-2\sqrt{-1}\partial\beta\qquad {\rm and}\qquad \mathcal H\circ {\mathcal I}^{-1}(v)=-2\sqrt{-1}\partial\beta+\sqrt{-1}d\beta \otimes t +\sqrt{-1}\beta\otimes dt.\]
By this, for $\Omega=\left(
\begin{array}{ccc}
0& \alpha_{1}&v  \\
0&    0 &\alpha_{2}\\
0&0&0 
\end{array}
\right)$, we have $(\Omega, V,W_{\ast},F^{\ast})\in {\rm Ob}({\rm Rep}({\frak n}, W_{\ast},F^{\ast}))$ and 
\[\Phi_{C}((\Omega, V,W_{\ast},F^{\ast}))=(V, W_{\ast} ,F^{\ast}, \omega,\omega^{\prime}, a_{\beta}).\]

\subsection{Functoriality}\label{funcc}
Let $M_{1}$ and $M_{2}$ be compact K\"ahler manifolds. Let $f : M_{2}\to M_{1}$ be a holomorphic map.
By the pull-back $f^{\ast}:A^{\ast}(M_{1})\to A^{\ast}(M_{2})$, we have the  morphism between $\mathbb R$-mixed Hodge diagrams
$\mathcal{D}(M_{1})$ and $\mathcal{D}(M_{2})$.

Take the canonical $1$-minimal models $\,_{1}\phi:\,_{1}{\mathcal M}^{\ast}\to A^{\ast}(M_{1})$, $\,_{2}\phi:\,_{2}{\mathcal M}^{\ast}\to A^{\ast}(M_{2})$,
 $\,_{1}\varphi:\,_{1}{\mathcal N}^{\ast}\to A^{\ast}(M_{1})$ and $\,_{2}\varphi:\,_{2}{\mathcal N}^{\ast}\to A^{\ast}(M_{2},)$ as in the last subsection.
By the argument in \cite[Section 6.10]{Kasuya}, we can extend 
$f^{\ast}:A^{\ast}(M_{1})\to A^{\ast}(M_{2})$ to the maps 
$f_{\mathcal M}:\,_{1}{\mathcal M}^{\ast}\to \,_{2}{\mathcal M}^{\ast}$ and $f_{\mathcal N}:\,_{1}{\mathcal N}^{\ast}\to \,_{2}{\mathcal N}^{\ast}$
such that  $f^{\ast}\circ \,_{1}\phi=\,_{2}\phi \circ f_{\mathcal M}$, 
 $f^{\ast}\circ \,_{1}\varphi=\,_{2}\varphi \circ f_{\mathcal N}$
,$f_{\mathcal M}$  preserves gradings and  $f_{\mathcal N}$ preserves  the bigradings.

Let  $x_{1}\in M_{1}$ and $x_{2}\in M_{2}$.
Take the isomorphism $\,_{1}{\mathcal I}:\,_{1}{\mathcal N}^{\ast}\to  \,_{1}{\mathcal M}^{\ast}_{\C}$ (resp. $\,_{2}{\mathcal I}:\,_{2}{\mathcal N}^{\ast}\to  \,_{2}{\mathcal M}^{\ast}_{\C}$)   and 
the homotopy  $\,_{1}{\mathcal H}:\,_{1}{\mathcal N}^{\ast}\to  A^{\ast}(M_{1})_{\C}\otimes [t,dt]$ (resp. $\,_{2}{\mathcal H}:\,_{2}{\mathcal N}^{\ast}\to  A^{\ast}(M_{2})_{\C}\otimes [t,dt]$) associated with the  the augmented $\mathbb R$-mixed Hodge diagram $(\mathcal{D}(M_{1}), \epsilon_{x_{1}},\epsilon_{x_{1}})$ (resp.  $(\mathcal{D}(M_{2}), \epsilon_{x_{2}},\epsilon_{x_{2}})$) as in Section  \ref{IHH}.
\begin{proposition}[{\rm cf.  \cite[Proposition 6.11.2]{Kasuya}}]
If $f(x_{1})=x_{2}$, then 
\[\,_{2}{\mathcal I}\circ f_{\mathcal N}=f_{\mathcal M}\circ \,_{1}{\mathcal I}\]
 and \[\,_{2}{\mathcal H}\circ f_{\mathcal N}=(f^{\ast}\otimes {\rm id}_{[t,dt]})\circ \,_{1}{\mathcal H}.\]
\end{proposition}
\begin{proof}
By the assumption, we have $ \epsilon_{x_{2}}\circ f^{\ast}= \epsilon_{x_{1}}$.
For each  $C_{i}={\rm ker}\epsilon_{x_{i} \vert A^{0}(M)}$, we consider $\delta_{C_{i}}$ as in  Section  \ref{IHH}.
Then we have $f^{\ast}\circ \delta_{C_{1}}=\delta_{C_{2}}\circ f^{\ast}$.
By this we can prove the proposition inductively by the essentially same argument in the proof of \cite[Proposition 6.11.2]{Kasuya}.
\end{proof}

This Proposition says that Morgan mixed Hodge structures themselves are functorial in contrast to Remark \ref{funn}.

\subsection{$\R$-NMHS and Hain-Zucker's theorem}
Let $M$ be a compact complex manifold admitting  a K\"ahler metric $g$.
Then we can easily say that  $(VMHS^{u}_{\R}(M), \epsilon_{x}, \kappa^{\prime})$ is an $\mathbb R$-NMHS.
Consider the fundamental group $\pi_{1}(M,x)$ at $x\in M$.
Then, by using Chen's  iterated integral, Hain constructed  a canonical $\R$-mixed Hodge structure $(W_{\ast},F^{\ast})$ on the completion $\widehat{\R\pi_{1}(M,x)}=\varprojlim \R\pi_{1}(M,x)/J^{r+1}$ of the group ring $\R\pi_{1}(M,x)$ associated with powers $J^{r+1}$ of augmentation ideal $J$ depending on the choice of a point $x\in M$  (see \cite{HainI}).
We define the category ${\rm Rep}(\widehat{\R\pi_{1}(M,x)}, W_{\ast}, F^{\ast})$ 
such that:
\begin{itemize}
\item Objects are $(V,W_{\ast}, F^{\ast}, \sigma)$
so that
 \begin{itemize}
\item $(W_{\ast},F^{\ast})$ is a $\R$-mixed Hodge structure on a finite-dimensional $\R$-vector space $V$.
\item $\sigma: \widehat{\R\pi_{1}(M,x)}\to {\rm End}(V)$ is a module structure which is a morphism of $\R$-mixed Hodge structures.
\end{itemize}
\item Morphisms are morphisms of $\widehat{\R\pi_{1}(M,x)}$-modules which are  morphisms of $\R$-mixed Hodge structures.
\end{itemize}
Define  the functors $\tau^{\prime}_{1}: {\rm Rep}(\widehat{\R\pi_{1}(M,x)}, W_{\ast},F^{\ast})\to  {\mathcal MHS}_{\mathbb K}$ and $\tau^{\prime}_{2}:   {\rm Rep}(\widehat{\R\pi_{1}(M,x)}, W_{\ast},F^{\ast})\to {\rm Rep}({\frak n}, W_{\ast},F^{\ast})$ such that 
\[\tau_{1}^{\prime}: {\rm Ob}({\rm Rep}(\widehat{\R\pi_{1}(M,x)}, W_{\ast},F^{\ast}))\ni (V, W_{\ast}, F^{\ast}, \sigma)\mapsto (V, W_{\ast}, F^{\ast})\in  {\rm Ob}({\mathcal MHS}_{\mathbb K})\] 
and
 \[\tau_{2}: {\rm Ob}({\mathcal MHS}_{\mathbb K})\ni (V, W_{\ast}, F^{\ast})\to (V, W_{\ast}, F^{\ast},0)\in {\rm Ob}({\rm Rep}(\widehat{\R\pi_{1}(M,x)}, W_{\ast},F^{\ast})).\]
Then $({\rm Rep}(\widehat{\R\pi_{1}(M,x)}, W_{\ast},F^{\ast}),\tau^{\prime}_{1},\tau^{\prime}_{2})$ is an $\R$-NMHS.
Now, Hain-Zucker's theorem in \cite{HZ} can be stated as following (see \cite[Proposition 5.3]{Ara}).
\begin{theorem}[\cite{HZ}]\label{hazu}
The $\mathbb R$-NMHS $(VMHS^{u}_{\R}(M), \epsilon_{x}, \kappa^{\prime})$ is equivalent to the $\R$-NMHS
\[ ({\rm Rep}(\widehat{\R\pi_{1}(M,x)}, W_{\ast},F^{\ast}),\tau^{\prime}_{1},\tau^{\prime}_{2}).\] 
More precisely the monodromy representation functor defines an equivalence $h_{x}:VMHS^{u}_{\R}(M)\to {\rm Rep}(\widehat{\R\pi_{1}(M,x)}$ between tensor categories and there exists a quasi-inverse $q_{x}$ of $h_{x}$ such that  the diagram
\[\xymatrix{
VMHS^{u}_{\R}(M)\ar@<1.0 mm>[rr]^{\epsilon_{x}}\ar[d]^{h_{x}}&&\ar@<1.0mm>[ll]^{\kappa^{\prime}}{\mathcal MHS}_{\mathbb K}\ar[d]^{=}\\
({\rm Rep}(\widehat{\R\pi_{1}(M,x)}), W_{\ast},F^{\ast})\ar@<2.0 mm>[u]^{q_{x}}\ar@<1.0 mm>[rr]^{\tau^{\prime}_{1}}&&\ar@<1.0 mm>[ll]^{\tau^{\prime}_{2}}{\mathcal MHS}_{\mathbb K}
}
\]
commutes.
\end{theorem}

Hain and Zucker proved this type equivalence for graded-polarizable "good" unipotent VMHSs over quasi-K\"ahler manifolds.
Goodness is a condition at infinity and the  graded-polarizablity is only used for studying local monodromy near to normal crossing divisors.
Thus we may omit them on compact K\"ahler manifolds.

We can see  an explicit connection between Sullivan's $1$-minimal model and Chen's iterated integral at the level of complex geometry.
By Theorem \ref{hazu} and Theorem \ref{ppptNH}, we have the following result.

\begin{theorem}\label{nmht}
The $\R$-NMHS $({\rm Rep}({\frak n}, W_{\ast},F^{\ast}),\tau_{1},\tau_{2}) $ associated with a Morgan mixed Hodge structure of an augmented  $\R$-mixed-Hodge diagram  $(\mathcal{D}(M), \epsilon_{x},\epsilon_{x})$ 
 is  equivalent to
 the $\R$-NMHS $({\rm Rep}(\widehat{\R\pi_{1}(M,x)}, W_{\ast},F^{\ast}),\tau^{\prime}_{1},\tau^{\prime}_{2})$.

\end{theorem}

\end{document}